\documentclass{article}

\usepackage{amsthm}
\usepackage{amssymb}
\usepackage{amsfonts}
\usepackage{enumerate}
\usepackage{amsmath}
\usepackage{mathabx}
\usepackage[utf8]{inputenc}
\usepackage{tikz}

\renewcommand{\l}{\lambda}
\renewcommand{\d}{\delta}

\newtheorem{thm}{Theorem}
\newtheorem{prop}[thm]{Proposition}

\newtheorem{lemma}[thm]{Lemma}

\newtheorem{question}[thm]{Question}

\theoremstyle{definition}
\newtheorem{example}[thm]{Example}

\title{Phase Transition for a Family of Complex-driven Loewner Hulls }
\author{Joan Lind and Jeffrey Utley}

\begin{document}

\maketitle

\begin{abstract}
Building on H. Tran's study of Loewner hulls generated by complex-valued driving functions, which showed the existence of a phase transition \cite{T},  
we answer the question of whether the phase transition for complex-driven hulls matches the phase transition for real-driven hulls.
This is accomplished through a detailed study of the Loewner hulls generated by driving functions $c\sqrt{1-t}$ and $c\sqrt{\tau + t}$ for $c \in \mathbb{C}$ and $\tau \geq 0$.
This family also provides examples of new geometric behavior that is possible for complex-driven hulls but prohibited for real-driven hulls.
\end{abstract}

\vspace{0.1in}


\vspace{0.1in}


\tableofcontents

\section{Introduction and results}
The Loewner differential equation forms the deterministic foundation for Schramm-Loewner evolution (SLE), which was
introduced by O.~Schramm in 2000 \cite{SLEintro} and led to breakthroughs in probability and statistical physics.  
Due to the significant role it plays in SLE, the  Loewner equation became a focus of study in its own right.  
Roughly, the Loewner equation provides a correspondence between certain families of 2-dimensional sets, such as curves in the plane, and real-valued functions.
The real-valued function (called the {\it driving function}) uniquely determines the growing family of sets (called {\it hulls}).

In \cite{T}, H.~Tran began the first study of the Loewner hulls generated by complex-valued driving functions.  
While many of the familiar properties from the real-valued setting carry over to the complex-valued setting, 
one immediately encounters differences.  
In particular, the Loewner maps $g_t$ no longer map into a standard domain,
but rather there are left hulls $L_t$ and right hulls $R_t$ so that
$$g_t : \mathbb{C} \setminus L_t \to \mathbb{C} \setminus R_t.$$

One of Tran's main results (stated below) is the complex-valued version of the result that the Loewner hulls generated by real-valued driving functions with small Lip$(1/2)$ norm are quasi-arcs \cite{MR, L}, 
which implies that they are simple curves.
Note that the Lip$(1/2)$ norm of $\l$ is the smallest $c$ such that 
$$ |\l(t) - \l(s)| \leq c |t-s|^{1/2}$$
for all $t,s$ in the domain of $\l$.

\begin{thm}[\cite{T}] \label{HuyTheorem}
There exists $\sigma >0$ so that when
 $\l:[0,T] \to \mathbb{C}$ has Lip$(1/2)$ norm less than $\sigma$,
then  $L_t = \gamma[-t,t]$ for a quasi-arc $\gamma: [-T, T] \to \mathbb{C}$.
Moreover, 
\begin{equation}\label{Lcurve}
\gamma(t) = \lim_{y \to 0^+} g_t^{-1}(\l(t)+iy) \;\; \text{ and } \;\; \gamma(-t) = \lim_{y \to 0^-} g_t^{-1}(\l(t)+iy).
\end{equation}
\end{thm}

The optimal value of $\sigma$ in this theorem is not known.  Tran's work shows that $ \sigma \geq 1/3 $, and we also know that $\sigma \leq 4$ from the real-valued case \cite{L}.   This suggests the following question:

\begin{question} \label{Q}
Is $\sigma = 4$, matching the optimal norm in the real-valued driving function case?
\end{question}

We show that the answer to this question is `no' and $\sigma < 3.723$.  
This is accomplished through a detailed study of the hulls generated by $c\sqrt{1-t}$ for $c \in \mathbb{C}$, which  
 extends the real-valued case computed in \cite{KNK}.

\begin{thm}\label{familyTHM1}
Let $c \in \mathbb{C} \setminus\{\pm 4\}$, let $\l(t) = c\sqrt{1-t}$ be the Loewner driving function,  
and define $\alpha = \frac{1}{2}\left[ 1 - \frac{c}{\sqrt{c^2-16}}\right]$.
\begin{enumerate}
    \item[(i)] When Re$(\alpha) > 0$, the left hull $L_t$ is a simple curve for all $t \leq 1$ (defined as in \eqref{Lcurve}).
   When $c \notin i\mathbb{R}$, the curve spirals infinitely around its two endpoints $\gamma(1), \gamma(-1)$ as $t \to 1$.
    \item[(ii)] When Re$(\alpha) < 0$, the left hull $L_t$ is a simple curve for $t<1$ (defined as in \eqref{Lcurve}), but $L_1$ consists of a simple closed curve and its interior.
    \item[(iii)]  When Re$(\alpha) = 0$, there is a time $t_c < 1$ so that the left hull $L_t$ is a simple curve for $t<t_c$ (defined as in \eqref{Lcurve}),
    $L_t$ is a non-simple curve for $t_c \leq t <1$,
    and $L_1$ consists of a curve and its interior.
\end{enumerate}

\end{thm}

\begin{figure}
\centering
\includegraphics[scale=0.6]{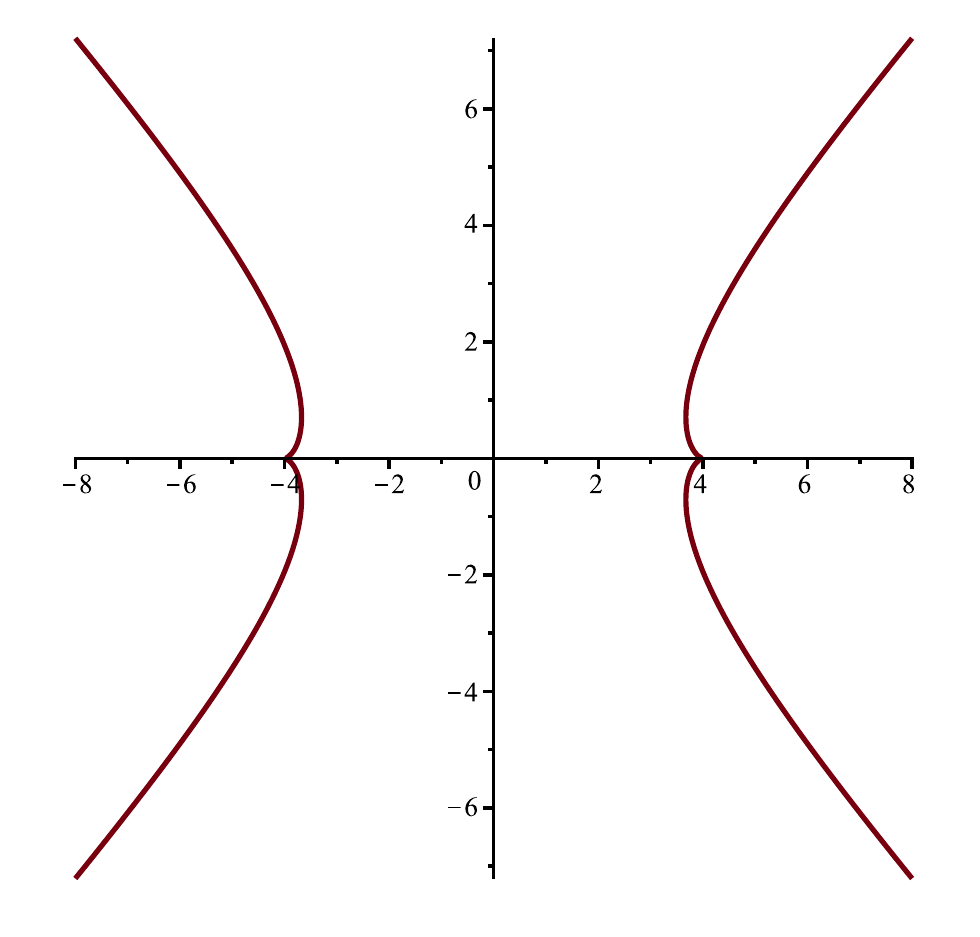}
\caption{
The values of $c \in \mathbb{C} \setminus \{ \pm 4 \}$ satisfying Re$(\alpha) =0$. 
%The minimal value of $|c|$ on these curves is approximately 3.722.
}\label{ReAlphaIs0}
\end{figure}

In Figure \ref{ReAlphaIs0}, we show the values of $c\in \mathbb{C} \setminus \{ \pm 4 \}$ with Re$(\alpha) =0$.  Since the minimal value of $|c|$ on these curves is approximately 3.722, Theorem \ref{familyTHM1} provides our answer to Question \ref{Q}.
Figure \ref{Examples} shows examples of the Re$(\alpha) >0$ case and the Re$(\alpha) <0$ case from Theorem \ref{familyTHM1}.
In both images, the blue portion of the curve corresponds to $\gamma(t)$ for $t>0$, and the red portion corresponds to $\gamma(t)$ for $t<0$.  
On the left, $L_1$ is a simple curve  that spirals around its endpoints $A,B$ when $c=3.31+1.15i$.  
On the right, we see that $L_1$ is not simple curve for $c=5+2i$.  Rather as $t \to 1$, we have that $\gamma(t), \gamma(-t)$ both approach $B$ creating a simple loop.  The interior of this loop is also part of the time-1 hull.

\begin{figure}
\centering
\includegraphics[scale=0.53]{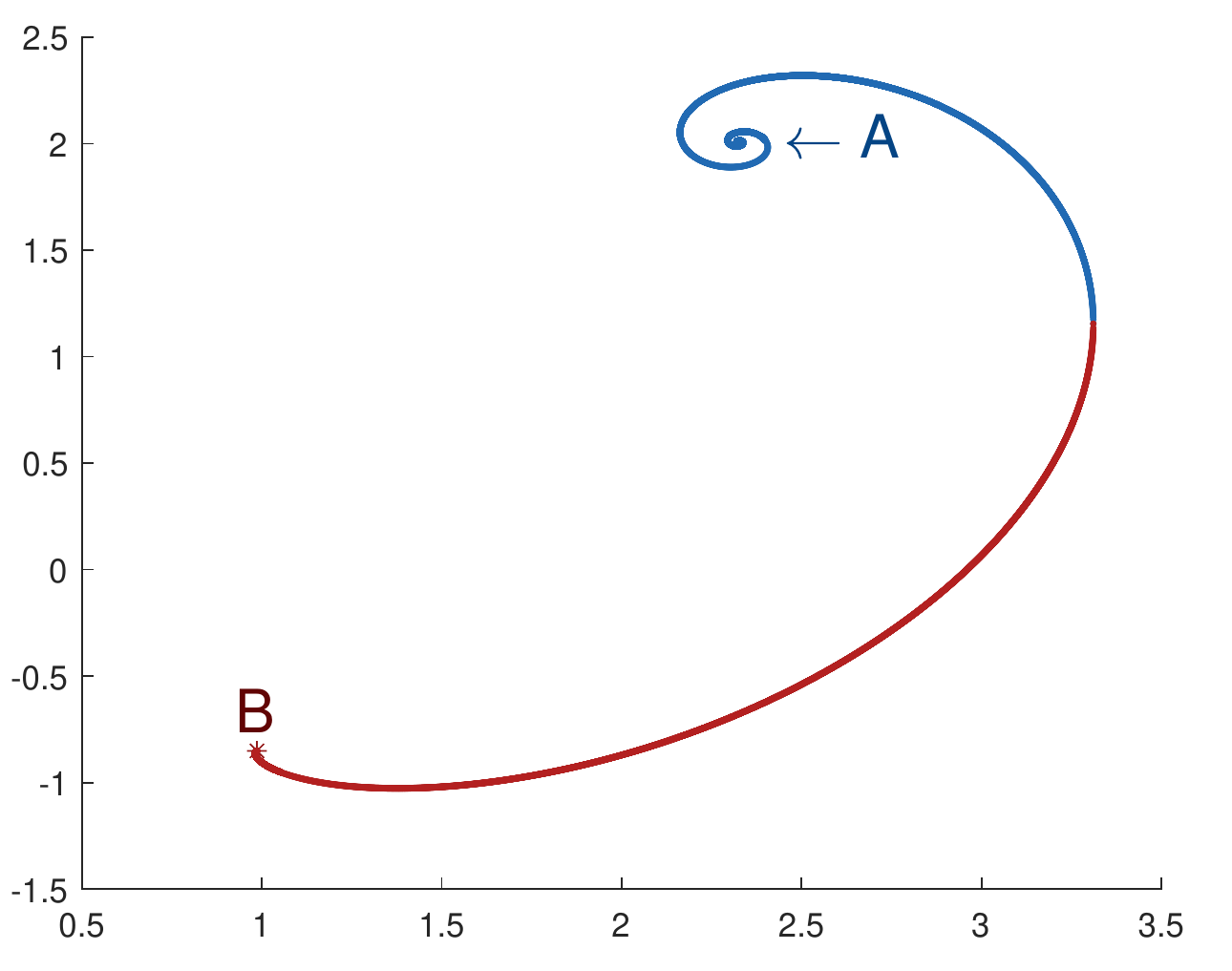}
\hfill
\includegraphics[scale=0.53]{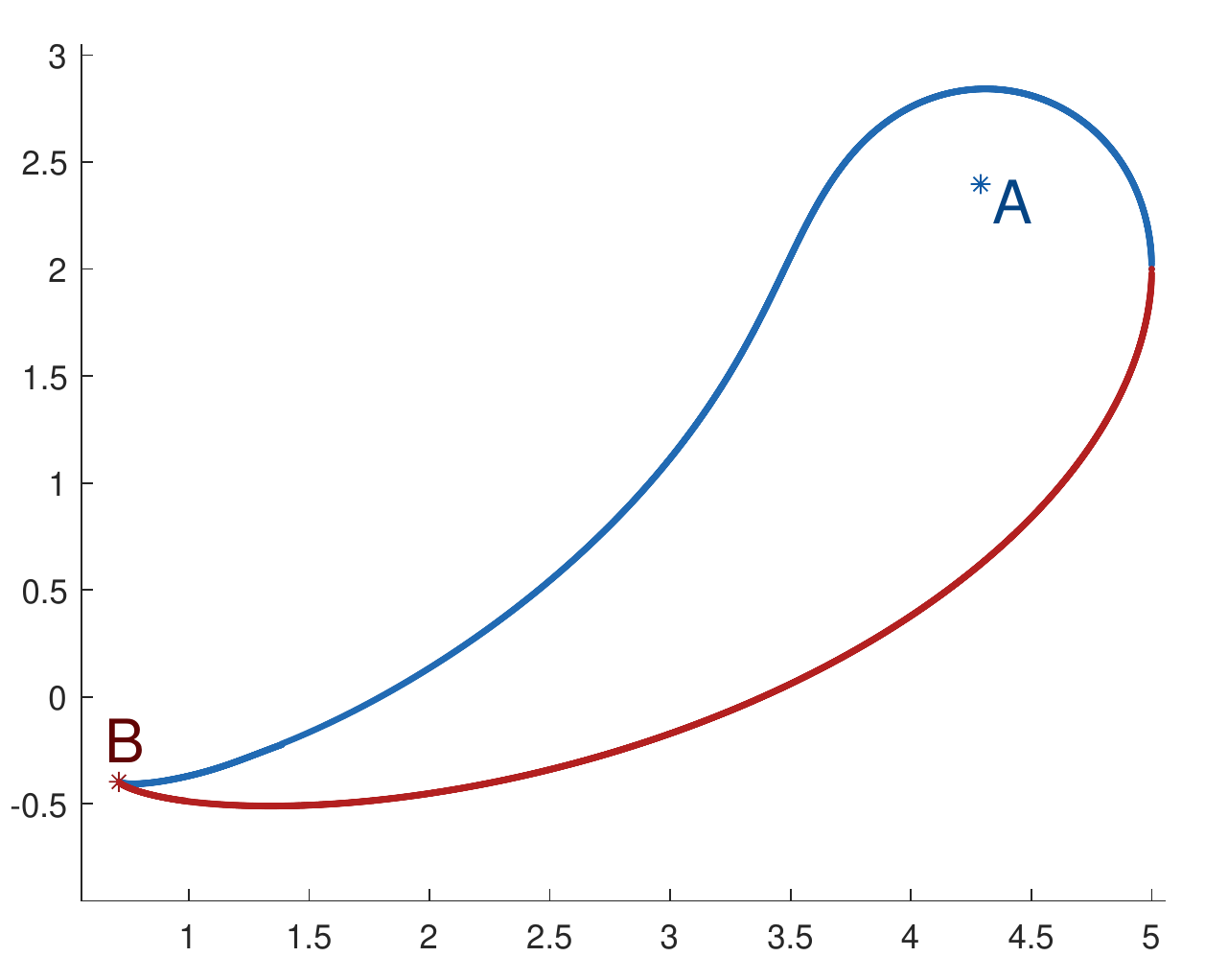}
\caption{Left: The left hull $L_1$ generated by $(3.31+1.15i)\sqrt{1-t}$ is a simple curve spiraling around its endpoints $A,B$.
Right: The left hull $L_1$ generated by $(5+2i)\sqrt{1-t}$ is the simple closed curve and its interior.
}\label{Examples}
\end{figure}

The third case of Theorem \ref{familyTHM1} is intriguing, because this behavior  is not possible in the real-valued driving function setting.  
Roughly, the left hull begins as a simple curve with two growing ends, and then at some time $t_c<1$, one end hits back on itself at $c$, stopping its growth, while the other end continues to grow for $t \in (t_c, 1]$.  In the range $[t_c, 1)$ the domain of the Loewner map is no longer connected, but contains two components.  
See Figure \ref{ReAlphaZeroExample}.

We also analyze the left hulls generated by driving functions $c\sqrt{\tau+t}$ for $c \in \mathbb{C}$ and $\tau\geq0$.

\begin{thm}\label{familyTHM2}
Let $c \in \mathbb{C} \setminus\{\pm 4i\}$, let $\tau \geq 0$, let $\l(t) = c\sqrt{\tau+t}$ be the Loewner driving function,  
and define $\delta = \frac{1}{2}\left[1-\frac{c}{\sqrt{c^2+16}}\right]$.
\begin{enumerate}
    \item[(i)] When $\tau =0$ and Re$(\delta) > 0$, the left hull $L_t$ is the union of two line segments emanating from the origin.
    \item[(ii)] When $\tau =0$ and Re$(\delta) \leq 0$, the left hull $L_t$ is one line segment emanating from the origin.
    \item[(iii)] When $\tau >0$ and Re$(\delta) \neq 0$, the left hull $L_t$ is a simple curve (defined as in \eqref{Lcurve}).
    \item[(iv)] When $\tau >0$ and Re$(\delta) = 0$, for $t$ large enough the left hull $L_t$ is a non-simple curve.
\end{enumerate}

\end{thm}

The fourth case represents another interesting difference between the complex setting and real setting.  In particular, we have a differentiable driving function that does not generate a simple curve hull, something that is not possible in the real setting. 
As we will see in our work, there are more challenges in the complex setting with turning local results into global results using the concatenation property.

So far, we have not mentioned the right hulls $R_t$.  
However, by the duality property (described in Section 2),
the right hulls generated by $c\sqrt{1-t}$ are related to the left hulls generated by driving functions of the form
$\hat{c}\sqrt{\tau+t}$, and vice versa.  
Thus taken together, Theorems \ref{familyTHM1} and \ref{familyTHM2} give us a complete understanding of the right and left hulls for both families of driving functions.

In fact, the Loewner flow gives a dynamic way to change $L_t$ into $R_t$.  
In our context, this means that we can morph the hull $L_1$ driven by $c\sqrt{1-t}$ (i.e. a Theorem \ref{familyTHM1} type hull) into $R_1$ which equals a rotated left hull driven by $-ic\sqrt{t}$ (i.e. a Theorem \ref{familyTHM2}(i)-(ii) type hull).
The intermediate hulls that we see along the way are a union of scaled Theorem \ref{familyTHM1} type hulls and rotated Theorem \ref{familyTHM2}(iii)-(iv) type hulls.
This viewpoint, illustrated in Figure \ref{FlowExample}, ties together all the hulls of Theorem \ref{familyTHM1} and Theorem \ref{familyTHM2}.

\begin{figure}
\centering
\includegraphics[scale=0.195]{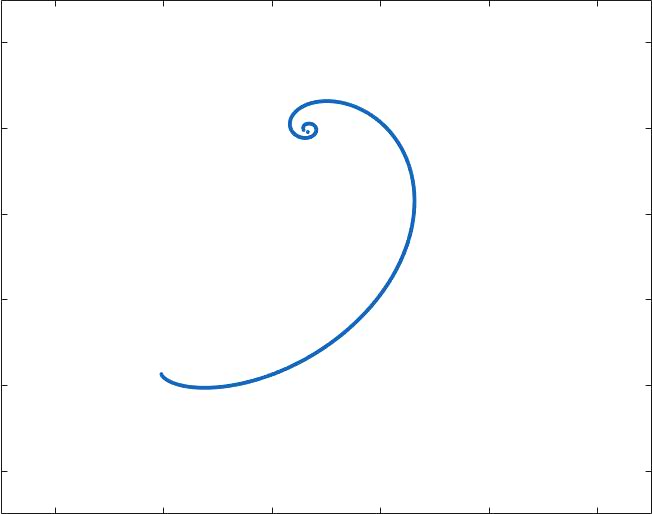}
\includegraphics[scale=0.195]{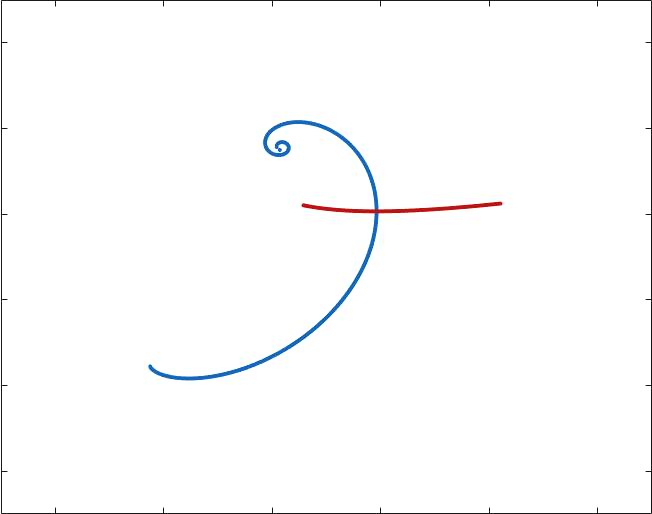}
\includegraphics[scale=0.195]{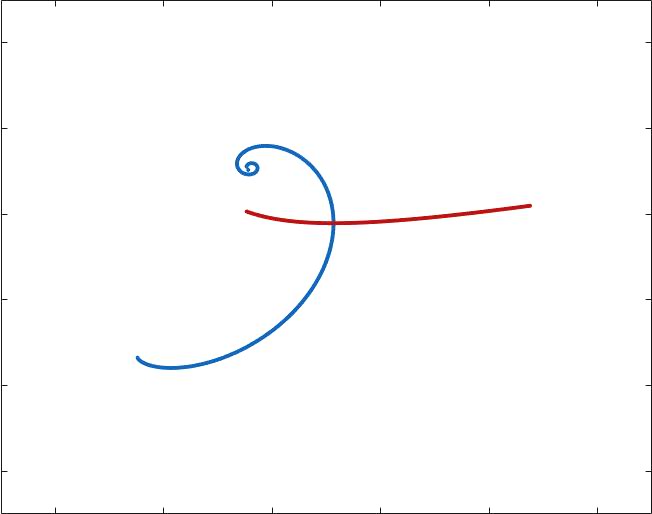}

\includegraphics[scale=0.195]{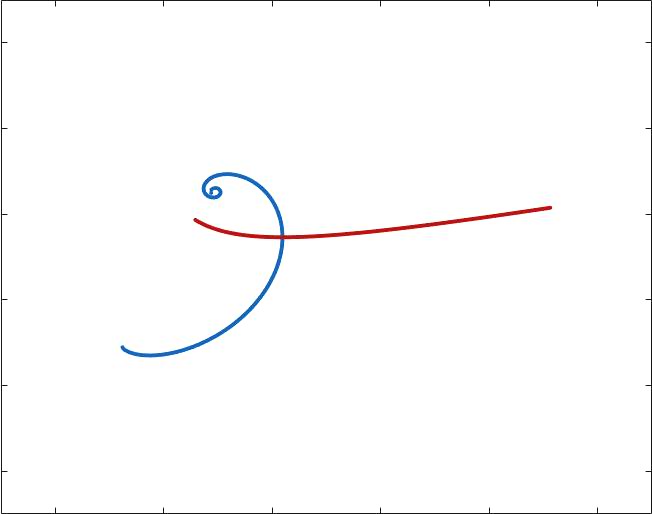}
\includegraphics[scale=0.195]{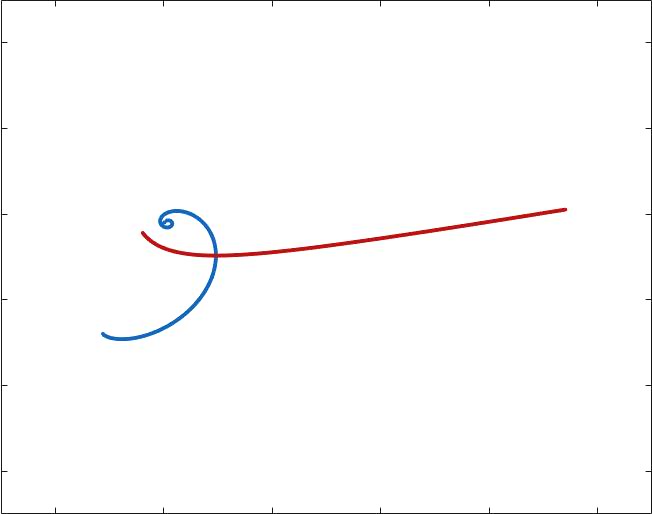}
\includegraphics[scale=0.195]{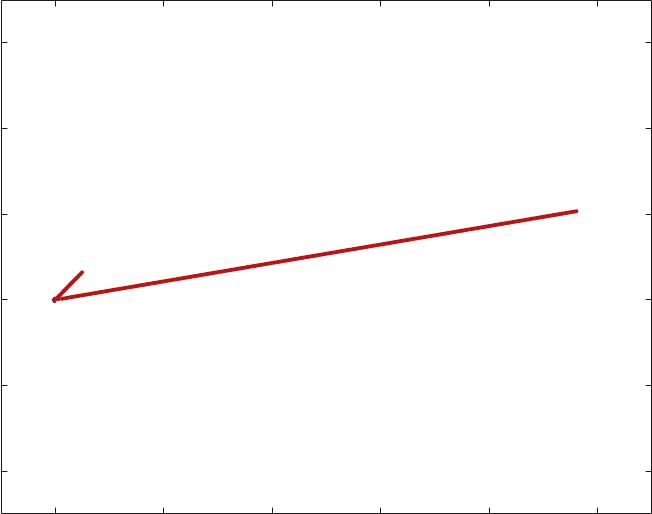}

\caption{The Loewner flow changes $L_1$ driven by $(3.31+1.15i)\sqrt{1-t}$ (top left image) into $R_1$ (bottom right image).  
The intermediate images show  
$g_t(L_1 \setminus L_t)$ in blue and  $R_t$ in red for $t= 0.196, 0.392, 0.588, 0.784$.
}\label{FlowExample}
\end{figure}

We give a brief description of our approach and organization.
Section \ref{background} contains background information on the complex Loewner equation.
In Section \ref{OneMinusT}, we study the left hulls generated by driving functions $c\sqrt{1-t}$, focusing on the cases away from the phase transition. 
We begin by calculating an implicit solution for the Loewner map $g_t$, following the method of \cite{KNK}.
Then using Tran's Theorem \ref{HuyTheorem} and the concatentation property, we show that the left hulls $L_t$ are simple curves before time 1.  
We use the implicit solution and a holomophic motion argument to analyze the behavior of these curves as time approaches 1. 
Finally, we determine the time-1 hulls $L_1$.
In Section \ref{tauplust}, we focus on the left hulls generated by driving functions $c\sqrt{\tau+t}$.  We calculate an implicit solution, and then separately address the cases when $\tau >0$ and when $\tau =0$.
The last section contains a discussion of the hulls associated with the phase transition, i.e. the cases Theorem \ref{familyTHM1}(iii) and Theorem \ref{familyTHM2}(iv).

When the driving function is $c\sqrt{t}$, our computation in Section \ref{tauplust} yields an explicit solution for the inverse Loewner map $g_t^{-1}$.  
By using this solution along with
 the well-known algorithm suggested by D. Marshall and S. Rohde in \cite{MR} for approximating the hulls of real-valued driving functions,
we created a program to simulate the hulls for complex-valued driving functions.
In particular, the images in Figures \ref{Examples} and \ref{FlowExample}  were generated from this program.

% Backgroud Section %
\section{Background}\label{background}

We begin by briefly reviewing  notation, terminology, and results associated with the Loewner equation driven by complex-valued functions from \cite{T}.

For $\l : [0,T] \to \mathbb{C}$, the complex Loewner equation is the following IVP:
\begin{equation}\label{cle}
\partial_t g_t(z) = \frac{2}{g_t(z) - \l(t)}, \;\;\; g_0(z) = z.
\end{equation}
The lifetime of a point $z$ is 
\begin{equation*}
T_z  = \sup \{ t \, : \,  g_t(z) \neq \l(t) \},
\end{equation*}
with $T_{\l(0)}=0$ and $T_z = \infty$ if $g_t(z) \neq \l(t)$ for all $t \in [0,T]$.
When $T_z < \infty$, we will say that $z$ is captured by $\l$ at time $T_z$.
The left hull $L_t$ and right hull $R_t$ are defined as
$$ L_t = \{ z \in \mathbb{C}  \, : \, T_z \leq t \}
\;\;\; \text{ and } \;\;\; R_t = \mathbb{C} \setminus g_t(\mathbb{C} \setminus L_t).$$
Then  $g_t$ is a conformal map from $\mathbb{C} \setminus L_t$ onto $\mathbb{C} \setminus R_t$. 
Further, by setting $g_t(\infty) = \infty$, the map $g_t$ can be extended to be conformal in a neighborhood of infinity. 

When we need to show the dependence of  the hulls on the driving function $\l$, we will use the notation $L_{t, \l}$ and $R_{t, \l}$.  
With this notation, we can state the following important properties of the hulls:
\begin{itemize}
\item Translation Property:  For $a \in \mathbb{C}$, then 
$$L_{t, \l +a} = a + L_{t,\l} \;\;\; \text{ and } \;\;\; 
R_{t, \l +a} = a + R_{t,\l}.$$

\item Scaling Property:  For $a >0$, then 
$$L_{t, a\l(\cdot/a^2)} = a L_{t/a^2, \l} \;\;\; \text{ and } \;\;\;
R_{t, a\l(\cdot/a^2)} = a R_{t/a^2, \l}$$
where the notation $\l(\cdot/a^2)$ indicates the map $s \mapsto \l(s/a^2)$.

\item Reflection Property:  Let Ref$_\mathbb{R}(K)$ be the reflection of a set $K$ over the real axis and Ref$_{i\mathbb{R}}(K)$ be the reflection over the imaginary axis.  
Then 
$$L_{t, \overline{\l}} = \text{Ref}_\mathbb{R}(L_{t,\l}),   \;\;\; L_{t, -\overline{\l}} = \text{Ref}_{i\mathbb{R}}(L_{t,\l}), \;\;\; \text{ and } \;\;\;
L_{t, -\l} = - L_{t,\l}.$$ 
(Note that the third property is a consequence of the first two.)
The same three properties also hold for the right hull.

\item Concatenation Property:  
$$L_{t+s, \l} = L_{t, \l} \cup g_t^{-1} \left( L_{s, \l(t+\cdot)} \setminus R_{t, \l} \right) \;\;\; \text{ and } \;\;\;
g_t\left( L_{t+s, \l} \setminus L_{t,\l} \right) = L_{s, \l(t+ \cdot)} \setminus R_{t, \l}.$$

\item Duality Property:  
$$L_{t, \l} = i R_{t, -i\l(t-\,\cdot)} \;\;\; \text{ and } \;\;\; R_{t, \l} = i L_{t, -i\l(t- \, \cdot)}.$$
\end{itemize}

When $\l$ is real-valued, the first reflection property shows that $L_t$ is symmetric with respect to $\mathbb{R}$.  
Thus, the hull $L_t$ driven by a real-valued function can be understood by simply studying  $K_t = L_t \cap \mathbb{H}$, which is the well-studied chordal Loewner hull in $\mathbb{H}$ driven by $\l$.
The right hull in the real-valued case will always be an interval of the real line.

The translation, scaling, and reflection properties are natural extensions of these well-known properties for the real-valued driving function case.  The concatenation property looks and behaves slightly different than its real-valued counterpart, due to the reference to the right hull $R_{t, \l}$.  Because of this difference, we wish to spend a little time discussing and justifying it.

\begin{proof}[Proof of Concatenation Property]
We begin by proving the second statement.  
By definition,
\begin{equation*}
L_{t+s, \l} \setminus L_{t, \l} = 
\{ z \in \mathbb{C}  \, : \, t < T_z \leq t+s \}.
 \end{equation*}
In other words, these are the points that have survived up until time $t$ but will be captured by time $t+s$.  
 Let $z \in L_{t+s, \l} \setminus L_{t, \l}.$  
 Since $z$ is not captured by time $t$,  $g_t(z)$ is a well-defined point in the range of $g_t$, implying that  $g_t(z) \notin R_t$.  
 Since $z$ is captured by time $t+s$, we must have that $g_t(z)$ is captured by time $s$ by the time-shifted driving function $\l(t+ \cdot)$, or equivalently $g_t(z) \in L_{s, \l(t+\cdot)}$.
 Thus $g_t(z) \in L_{s, \l(t+\cdot)} \setminus R_{t,\l}$, showing that
 $g_t(L_{t+s, \l} \setminus L_{t, \l} ) \subset L_{s, \l(t+\cdot)} \setminus R_{t,\l}$.
 Suppose that $w \in L_{s, \l(t+\cdot)} \setminus R_{t,\l}$.  
 Then $w$ is in the range of $g_t$, and so there is $z \in \mathbb{C}\setminus L_{t,\l}$ with $w = g_t(z)$.  
 The fact that $g_t(z)$ is captured at time $s$ by driving function $\l(t+\cdot)$ means that $z$ is captured at time $t+s$ by driving function $\l$.  So $z \in L_{t+s, \l} \setminus L_{t, \l}$, and $w\in g_t( L_{t+s, \l} \setminus L_{t, \l})$, proving that 
  $$g_t(L_{t+s, \l} \setminus L_{t, \l} ) = L_{s, \l(t+\cdot)} \setminus R_{t,\l}.$$
  
From this we obtain that 
  $$L_{t+s, \l} \setminus L_{t, \l}  = g_t^{-1}\left( L_{s, \l(t+\cdot)} \setminus R_{t,\l} \right),$$
and so
\begin{align*}
L_{t+s, \l}   &= \{ z \in \mathbb{C}  \, : \, T_z \leq t \} \cup \{ z \in \mathbb{C}  \, : \, t < T_z \leq t+s \} \\
& = L_{t,\l} \cup L_{t+s, \l} \setminus L_{t, \l} \\
&=  L_{t,\l} \cup g_t^{-1}\left( L_{s, \l(t+\cdot)} \setminus R_{t,\l} \right).
\end{align*}
\end{proof}

The difference between the concatenation property in the complex-valued case compared with that in the real-valued case is the following:  In the real-valued case, 
$$g_t(K_{t+s, \l} \setminus K_{t, \l} ) = K_{s, \l(t+\cdot)},$$ but in the complex case 
$$g_t(L_{t+s, \l} \setminus L_{t, \l} ) \neq L_{s, \l(t+\cdot)}.$$
Rather $L_{s, \l(t+\cdot)}$ and $R_{t, \l}$ will at least overlap at the point $\l(t)$, but the overlap can be even larger as we will see in the following example.

To illustrate the complex Loenwer equation, we briefly describe two examples.

\begin{example}
For the first example, let $\l(t) = 3\sqrt{2} \sqrt{1-t}$ for $t \in [0,1]$.  
One can compute that $L_1$ is the closed disk of radius $\sqrt{2}$ centered at $2\sqrt{2}$,
$\, R_1$ is the real interval $[0,4\sqrt{2}]$, and   
the conformal map $g_1: \mathbb{C} \setminus L_1 \to \mathbb{C} \setminus R_1$
is $g_1(z) = z + \frac{2}{z-2\sqrt{2}}$.
When $t <1$, then it can be shown that $L_t$ is a circular arc
and $R_t$ is a real interval containing $\l(t)$, as shown in Figure \ref{RealCircleExample}.  
In particular, there is an increasing odd function $\theta_t$ with $\theta_0 = 0$ and $\theta_1 = \pi$ so that $L_t = \gamma[-t,t]$ for
$\gamma(t) =  \sqrt{2} e^{i\theta_t} +2\sqrt{2} $.  
Under $g_t$, the two tips $\gamma(t), \gamma(-t)$ are mapped to $\l(t)$.  One can view $g_t$ as unzipping $L_t$ and flattening it, so the top curve of $L_t$ corresponds to the top of $R_t$, and similarly for the bottom.
As $t$ increases to 1, the left endpoint of $R_t$ decreases to 0 and the right endpoint increases to $4\sqrt{2}$.

%% FIGURE %%
\begin{figure}
\centering
\begin{tikzpicture}

\draw[thick] (0,0) to [out=90,in=0] (-1,1);
\draw[thick] (-1,1) to [out=180,in=45] (-1.707, 0.707);
\draw[fill] (-1.707,0.707) circle [radius=0.05];
\node[left] at (-1.8, 0.8) {$\gamma(t)$};
\draw[thick] (0,0) to [out=-90,in=0] (-1,-1);
\draw[thick] (-1,-1) to [out=180,in=-45] (-1.707, -0.707);
\draw[fill] (-1.707,-0.707) circle [radius=0.05];
\node[left] at (-1.8, -0.8) {$\gamma(-t)$};

\draw[fill] (0,0) circle [radius=0.05];
\node[right] at (0,0) {$3\sqrt{2}$};
\node[right] at (-0.7, -1.2) {$L_t$};

\draw[->] (1.7,1.8) to [out=45,in=135] (3.3,1.8);
\node[above] at (2.5,2.2) {$g_t$};

\draw[thick] (5,0) -- (8,0);

\draw[fill] (8,0) circle [radius=0.05];
\draw[fill] (5,0) circle [radius=0.05];
\draw[fill] (5.5,0) circle [radius=0.05];
\node[below] at (5.5,0) {$\l(t)$};
\node[above] at (7, 0) {$R_t$};

\end{tikzpicture}
\caption{The hulls $L_t$ and $R_t$ for driving function $\l(t) = 3\sqrt{2} \sqrt{1-t}$, when $t < 1$.  
}\label{RealCircleExample}
\end{figure}
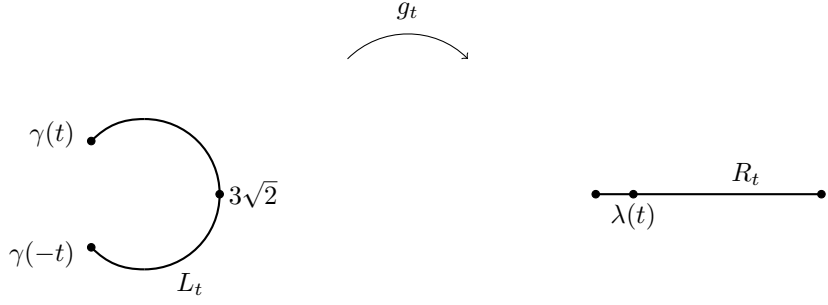

We now wish to consider the concatenation property in the context of this  example.  
Fix $t \in (0,1).$
The time-shifted driving function $\l(t+ \cdot)$ is given by:
$$ s \mapsto \l(t+s) = 3\sqrt{2} \sqrt{ 1- (t+s) }  = \sqrt{1-t} \cdot 3\sqrt{2} \sqrt{1-s/(1-t)}. $$
In other words, this is the map $s \mapsto a \l(s/a^2) $ for $a = \sqrt{1-t}$.
By the scaling property $L_{1-t,\l(t+ \cdot)}$ is a closed disc of radius $\sqrt{2}\sqrt{1-t}$, and its rightmost point will be $\l(t)$.
Thus the concatenation property tells us that 
$g_t(L_{1, \l} \setminus L_{t, \l})$ is the closed disc minus the left segment of $R_t$, as illustrated in Figure \ref{ConcatExample}. 
Notice that the sets
$g_t(L_{1, \l} \setminus L_{t, \l} )$ and $  L_{1-t, \l(t+\cdot)}$
differ by the left segment of $R_t$.
\end{example}

%%  FIGURE %%
\begin{figure}
\centering
\begin{tikzpicture}

\draw[fill, lightgray]  (-1,0) circle [radius = 1];
\draw[fill, lightgray]  (4.8,0) circle [radius = 0.7];

\draw[thick] (0,0) to [out=90,in=0] (-1,1);
\draw[thick] (-1,1) to [out=180,in=45] (-1.707, 0.707);
\draw[fill] (-1.707,0.707) circle [radius=0.05];
\node[left] at (-1.8, 0.8) {$\gamma(t)$};
\draw[thick] (0,0) to [out=-90,in=0] (-1,-1);
\draw[thick] (-1,-1) to [out=180,in=-45] (-1.707, -0.707);
\draw[fill] (-1.707,-0.707) circle [radius=0.05];
\node[left] at (-1.8, -0.8) {$\gamma(-t)$};

\draw[fill] (0,0) circle [radius=0.05];
\node[right] at (0,0) {$3\sqrt{2}$};
\node[right] at (-0.7, -1.2) {$L_t$};

\draw[->] (1.7,1.8) to [out=45,in=135] (3.3,1.8);
\node[above] at (2.5,2.2) {$g_t$};

\draw[thick] (5.05,0) -- (8,0);

\draw[fill] (8,0) circle [radius=0.05];
\draw[fill] (5.05,0) circle [radius=0.05];
\draw[fill] (5.5,0) circle [radius=0.05];
\node[below] at (5.75,0) {$\l(t)$};
\node[above] at (7, 0) {$R_t$};

\end{tikzpicture}
\caption{The left grey set, $L_1 \setminus L_t$, is a closed disc minus the black curve $\gamma[-t,t]$.  The right grey set is its conformal image, $g_t(L_1 \setminus L_t)$, which is a closed disc minus the left segment of $R_t$.   
}\label{ConcatExample}
\end{figure}
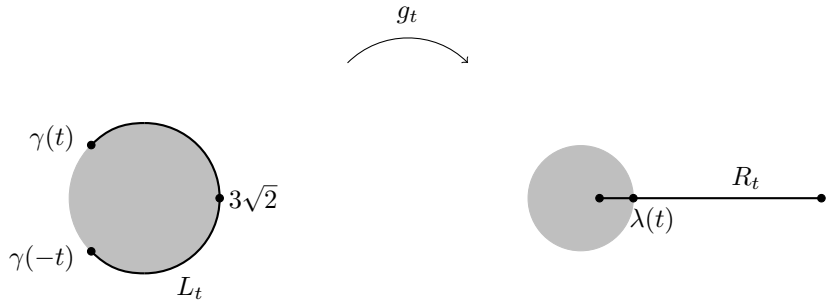

\begin{example}
For our next example, we consider the driving function 
$\l(t) = i \,3\sqrt{2}  \sqrt{t}$.  
By the duality property, we can identify the hulls for $\l$ by finding the hulls for
$$s \mapsto -i \l(t-s) = 3\sqrt{2} \sqrt{t-s} = 3 \sqrt{2t} \sqrt{ 1-s/t}. $$
However, we know these hulls from  our first example and the scaling property.
Thus
\begin{align*}
 L_{t, \l} &= i R_{t, -i\l(t-\,\cdot)} = [0, i\, 4\sqrt{2t}] \\
R_{t, \l} &= i L_{t, -i\l(t- \, \cdot)} = \text{ closed disc of radius } \sqrt{2t} \text{ centered at } i \,2\sqrt{2t}.
\end{align*}
Note that in this case, $L_{t, \l}$ is a simple curve, but it is only growing from one tip, instead from two tips as in the $t<1$ case of the first example.
\end{example}

When $c\in \mathbb{R}$, the upper-halfplane hulls $K_t = L_t \cap \mathbb{H}$ generated by driving functions $c\sqrt{1-t}$ and $c\sqrt{t}$ were computed in \cite{KNK}.  Applying Schwarz reflection yields the following.

\begin{prop}[\cite{KNK}] \label{KNKprop}
Let $c \in \mathbb{R}$.
\begin{enumerate}
    \item[(i)] When $|c|<4$, the left hull $L_t$ driven by $c\sqrt{1-t}$ is a simple curve $\gamma[-t,t]$ for all $t \leq 1$.  As $t \to 1$, the curve spirals infinitely around its two endpoints.
    \item[(ii)] When $|c| \geq 4$, the left hull $L_t$ driven by $c\sqrt{1-t}$ is a simple curve for $t<1$, but $L_1$ consists of a simple closed curve and its interior.
    \item[(iii)] The left hull $L_t$ driven by $c\sqrt{t}$ is the union of two line segments emanating from the origin.
\end{enumerate}
\end{prop}

\section{Left hulls driven by $c\sqrt{1-t}$}\label{OneMinusT}
In this section we study the left hulls driven by functions of the form $\l(t) = c\sqrt{1-t}$ for $c \in \mathbb{C}$. 
We begin by calculating an implicit solution to the Loewner equation driven by $\l$.

\subsection{Implicit solution}

For the driving function 
$$\lambda(t) = c\sqrt{1-t}\text{,} $$
with $c \in \mathbb{C}$, we may compute an implicit solution to \eqref{cle}, the complex Loewner equation,
following the approach in \cite{KNK}. 
Define the function
$$G = G(t,z) = \frac{g_t(z)}{\sqrt{1-t}}\text{.} $$
Then
\begin{align*}
\partial_t G = \partial_t \left(\frac{g_t(z)}{\sqrt{1-t}} \right) 
&= \frac{\left(\frac{2}{g_t(z)-c\sqrt{1-t}}\right)\sqrt{1-t}+ \frac{1}{2}\frac{g_t(z)}{\sqrt{1-t}}}{1-t} \\
&= \frac{1}{1-t}\left(\frac{2}{G-c}+\frac{G}{2}\right) \\
&= \frac{1}{2}\cdot\frac{1}{1-t}\left(\frac{G^2-cG+4}{G-c} \right)\text{.} 
\end{align*}
Since the
differential equation is separable, we obtain
\begin{equation}\label{separable}
    \int\frac{G-c}{G^2-cG+4}dG = \frac{1}{2}\int\frac{dt}{1-t} \text{.} 
\end{equation}
We wish to use a partial fraction decomposition for the integral on the left-hand side. 
The roots of $G^2 - cG + 4$ are
$$A = \frac{1}{2}[c+\sqrt{c^2-16}] \;\;\; \text{ and } \;\; 
B = \frac{1}{2}[c-\sqrt{c^2-16}]\text{.}$$
We will see later that these points will play a special role. 
Now, notice that
\begin{align}\label{AB4}
A + B &= \frac{1}{2}[c + \sqrt{c^2-16} + c - \sqrt{c^2-16}] = c 
 \text{,} \nonumber \\ 
A - B &= \frac{1}{2}[c+\sqrt{c^2-16}-(c-\sqrt{c^2-16})] =  \sqrt{c^2-16}\text{, and } \nonumber \\ 
AB &= \frac{1}{4}[(c+\sqrt{c^2-16})(c-\sqrt{c^2-16})] = \frac{1}{4}(c^2-(c^2-16))=4.
\end{align}
We assume that $c \neq \pm 4$, so that $A$ and $B$ are distinct.  
Since
$G^2 - cG + 4 = (G-A)(G-B) $, we can find the partial fractions decomposition:
$$\frac{G-c}{G^2 - cG + 4} = \frac{G-c}{(G-A)(G-B)}
= \frac{\alpha}{G-A} + \frac{\beta}{G-B} , $$
where $\alpha$ and $\beta$ must satisfy 
$$\alpha + \beta = 1 \;\;\; \text{ and } \;\;\;
\beta A + \alpha B = c\text{.} $$
By substitution,
$$(1-\alpha)A + \alpha B = c. $$
Thus
$$\alpha = \frac{c-A}{B-A} 
= \frac{B}{B-A} 
= \frac{\frac{1}{2}\left( c-\sqrt{c^2-16}\right)}{-\sqrt{c^2-16}} = \frac{1}{2}\left( 1 - \frac{c}{\sqrt{c^2-16}}\right) \text{,}$$
and
$$\beta = 1 - \alpha 
= \frac{1}{2}\left( 1+\frac{c}{\sqrt{c^2-16}}\right) \text{.} $$
Returning to  \eqref{separable}, we obtain
$$\int \left(\frac{\alpha}{G-A} + \frac{\beta}{G-B} \right)dG = \frac{1}{2}\int\frac{dt}{1-t}\text{.} $$
Integrating gives 
$$\alpha\log(G-A) + \beta\log(G-B) = -\log(\sqrt{1-t}) + C \text{,}$$
where $C$ denotes the constant of integration. 
Since  $\alpha + \beta = 1$,
\begin{align*}
C &= \alpha\log(G-A) + \beta\log(G-B) + (\alpha+\beta)\log(\sqrt{1-t}) \\
&= \alpha\log[(G-A)\sqrt{1-t}]+\beta\log[(G-B)\sqrt{1-t}] \\
&=\alpha\log(g_t(z)-A\sqrt{1-t}) + \beta\log(g_t(z)-B\sqrt{1-t}),
\end{align*}
where in the last step we used that $G = \frac{g_t(z)}{\sqrt{1-t}}$.
We now want to plug in our initial condition $g_0(z) =z $ to solve for the constant of integration $C$. Fixing $t = 0$, we find that
$$C = \alpha\log(z - A) + \beta\log(z - B).$$
Our equation then becomes
\begin{equation} \label{implicitSolution}
\alpha\log(g_t(z)-A\sqrt{1-t}) + \beta\log(g_t(z)-B\sqrt{1-t}) 
    =\alpha\log(z-A) + \beta\log(z-B)\text{.} 
\end{equation}

Now suppose that $\displaystyle z_t \in L_t \setminus \bigcup_{s \in [0,t)} L_s$.
This means that 
$g_t(z_t) = \lambda(t) = c\sqrt{1-t}$.
From \eqref{implicitSolution}, we find that $z_t$ must satisfy $$\alpha\log(c\sqrt{1-t}-A\sqrt{1-t}) + \beta\log(c\sqrt{1-t}-B\sqrt{1-t}) 
 =\alpha\log(z_t-A) + \beta\log(z_t-B)\text{.}  $$
Since $A+B =c$ and $\alpha+\beta = 1$, the left-hand side simplifies to
$ \alpha\log(B) + \beta\log(A) + \log\sqrt{1-t},$
and we have shown that
\begin{equation*}
    \alpha\log B + \beta\log A + \log\sqrt{1-t} = \alpha\log(z_t-A) + \beta\log(z_t-B)\text{.} 
\end{equation*}

We record the conclusions of our computation in the following lemma.

\begin{lemma} \label{ImplicitSolutionLemma}
For $c \in \mathbb{C} \setminus \{\pm 4\}$, let  $\l(t) = c\sqrt{1-t}$ be the driving function.  
Then for $t \in [0,1]$, the solution $g_t(z)$ to \eqref{cle} satisfies
\begin{equation*} 
\alpha\log(g_t(z)-A\sqrt{1-t}) + \beta\log(g_t(z)-B\sqrt{1-t}) 
    =\alpha\log(z-A) + \beta\log(z-B), 
\end{equation*}
and $\displaystyle z_t \in L_t \setminus \bigcup_{s \in [0,t)} L_s$
satisfies
\begin{equation}\label{TipEquation}
    \alpha\log B + \beta\log A + \log\sqrt{1-t} = \alpha\log(z_t-A) + \beta\log(z_t-B),
\end{equation}
where $A = \frac{1}{2}[c+\sqrt{c^2-16}], \,
B = \frac{1}{2}[c-\sqrt{c^2-16}], \,
\alpha = \frac{1}{2}\left[ 1 - \frac{c}{\sqrt{c^2-16}}\right],$
and $ \beta = \frac{1}{2}\left[ 1+\frac{c}{\sqrt{c^2-16}}\right].$
\end{lemma}

Our goal is to understand the hulls for the driving function $c\sqrt{1-t}$ for any $c \in \mathbb{C}$.
However, due to the reflection property, we may restrict our attention to $c$ with $\text{Arg}(c) \in [0, \frac{\pi}{2}]$. 
Under this restriction, we will look at the possible values for our parameters $A, B, \alpha,$ and $\beta$.

\begin{lemma} \label{parameterlemma}
Let $c \in \mathbb{C} \setminus \{\pm 4\}$ with $\text{Arg}(c) \in [0, \frac{\pi}{2}]$.
Then the parameters $A, B, \alpha$ and $\beta$ defined in Lemma \ref{ImplicitSolutionLemma} satisfy
\begin{align*}
    A &\in \{ z \, : \,  \text{Re}(z) \geq 0, \text{Im}(z) \geq 0
          \text{ and } |z| \geq 2 \}, \\ 
    B &\in \{ z \, : \,  \text{Re}(z) \geq 0, \text{Im}(z) \leq 0 
         \text{ and } |z| \leq 2\}, \\
    \alpha &\in \{z \, : \, \text{Re}(z) \leq 1/2 \text{ and Im}(z) \geq 0 \}, \;\; \text{ and } \\
    \beta &\in \{z \, : \, \text{Re}(z) \geq 1/2 \text{ and Im}(z) \leq 0 \}. 
\end{align*}
Moreover, if Arg$(c) \in (0, \pi/2)$, then the parameters will lie in the interior of the  given regions.
\end{lemma}

\begin{proof}
Because the formulas for our parameters $A, B, \alpha, \text{ and }\beta$ are conformal maps in $c$, the result will follow from the behavior of each map on the boundary. In particular, we will show that each parameter maps the boundary of the first quadrant, oriented counterclockwise, to the boundary of the desired set, oriented counterclockwise.
We separate the first quadrant boundary into three components:
(1) $c \in [0,4]$,
(2) $c \in (4, \infty),\,$ and
(3) $c \in i\mathbb{R}^+$.

We begin by studying $A= \frac{1}{2}[c+\sqrt{c^2-16}].$
Consider first the case $c \in [0,4]$. 
Since $c^2 - 16 \leq 0$, then
$A = \frac{c}{2}+\frac{i}{2}\sqrt{16-c^2}$. 
Therefore  Re$(A) = \frac{c}{2}$ and  Im$(A)=\frac{\sqrt{16-c^2}}{2}$ are both non-negative, and  
$$|A| = \sqrt{\frac{c^2}{4}+\frac{16-c^2}{4}}  = 2. $$
This means that $A$ maps $[0,4]$ to the quarter-circle of radius 2 centered at 0 in the first quadrant.
In the second case when $c \in (4, \infty)$, then  $A \in \mathbb{R}$, and $A$ increases from 2 to $\infty$ as $c$ increases from 4 to $\infty$.
In the last boundary case $c = ik$ for some $k > 0$,
$$A = i\frac{k}{2} + \frac{\sqrt{-k^2-16}}{2} = i\left( \frac{k}{2} + \frac{\sqrt{k^2+16}}{2} \right) \in i\mathbb{R}^+ .$$
 We also see that
$$|A| = \frac{k}{2}+\frac{\sqrt{k^2+16}}{2} > 2.  $$
Therefore $A$ maps the boundary of the first quadrant, oriented counterclockwise, to the boundary of the desired set, oriented counterclockwise.

Next we look at $B = \frac{1}{2}[c-\sqrt{c^2-16}]$, and we consider the same boundary cases. 
If $c \in [0,4]$, then Re$(B) = \frac{c}{2}$ and  Im$(B)=-\frac{\sqrt{16-c^2}}{2}$,
and it follows that Re$(B) \geq 0$, Im$(B) \leq 0$ and $|B| = 2$. 
In the boundary case $c \in (4, \infty)$, then $B \in \mathbb{R}$ and $B$ decreases from 2 to 0 as $c$ increases from 4 to $\infty$.
For the last boundary case, consider $c \in i\mathbb{R}^+$. Then $c = ik$ for $k > 0$ and 
$$B = i\left( \frac{k}{2}-\frac{\sqrt{k^2+16}}{2}\right) \in i \mathbb{R}^-. $$
Thus Im$(B)$ is negative and increases from -2 to 0 as $k$ increases from 0 to $\infty$.
We have shown that $B$ maps the boundary of the first quadrant, oriented counterclockwise, to the boundary of the desired set, oriented counterclockwise.

Our third parameter to study is $\alpha = \frac{1}{2}\left[ 1 - \frac{c}{\sqrt{c^2-16}}\right]$.
Consider the case that $c \in [0,4]$. 
Then
$$\alpha = \frac{1}{2}\left[1-\frac{c}{i\sqrt{16-c^2}}\right]=\frac{1}{2}+i\frac{c}{2\sqrt{16-c^2}} \text{.}$$
Thus Re$(\alpha) = \frac{1}{2}$ and Im$(\alpha)$ increases from 0 to $\infty$ as $c$ increases from 0 to 4.
When $c \in (4, \infty)$, then $\alpha$ is real-valued and $\alpha < 0$ since $c > \sqrt{c^2-16}$. Further, $\alpha$ increases from $-\infty$ to 0, as $c$ increases from 4 to $\infty$.
Lastly, we consider the boundary case that  $c = ik$ for $k > 0$.  Then
$$\alpha  =\frac{1}{2}\left[1-\frac{ik}{\sqrt{-k^2-16}} \right]
=\frac{1}{2}\left[1 -\frac{k}{\sqrt{k^2+16}} \right] \in \mathbb{R} \text{.}$$
Further, $\alpha$ increases from 0 to $\frac{1}{2}$ to 0 as $k$ decreases from $\infty$ to 0. 
Thus $\alpha$ maps the boundary of the first quadrant, oriented counterclockwise, to the boundary of the desired set, oriented counterclockwise.

Because $\alpha+\beta=1$, we may conclude Re$(\alpha)+\text{Re}(\beta)=1$ and Im$(\alpha)+\text{Im}(\beta) = 0$. Since Re$(\alpha) \leq \frac{1}{2}$, we must then have Re$(\beta) \geq \frac{1}{2}$. Since Im$(\alpha) \geq 0$, we must then have Re$(\beta) \leq 0$.
\end{proof}

\subsection{Simple curves before time 1}

In this section, we prove that the left hulls $L_s$ for $s<1$ generated by $c\sqrt{1-t}$ are simple curves when $c$ is not on the phase transition curve Re$(\alpha) =0$.  
Note that this implies that if the hull $L_1$ is not a simple curve for Re$(\alpha) \neq 0$, then the non-simpleness arises at time $t=1$.

\begin{prop}  \label{simplecurve1}
Let $c \in \mathbb{C}$,  let $\alpha = \frac{1}{2}\left[ 1 - \frac{c}{\sqrt{c^2-16}}\right]$,
let $\lambda(t) = c\sqrt{1-t}$, and
let $ s \in (0, 1)$. 
\begin{enumerate}
\item[(i)]When  Re$(\alpha) \neq 0$, 
the left hull $L_s$ generated by $\l |_{[0,s]}$ is a simple curve.
In particular, $L_s = \gamma[-s,s]$ for a simple curve $\gamma: [-s, s] \to \mathbb{C}$ with
$$\gamma(t) = \lim_{y \to 0^+} g_t^{-1}(\l(t)+iy) \;\; \text{ and } \;\; \gamma(-t) = \lim_{y \to 0^-} g_t^{-1}(\l(t)+iy).$$

\item[(ii)] When Re$(\alpha) = 0$ and Arg$(c) \in (0, \frac{\pi}{2})$, 
 the left hull $L_s$ generated by $\l |_{[0,s]}$ can be decomposed into two sets, an upper left hull $L_s^+$ and a lower left hull $L_s^-$
with $L_s^+ \cup L_s^- = L_s$ and $L_s^+ \cap L_s^- = \{c\}$.  
The lower left hull is a simple curve with $L_s^- = \gamma[-s,0]$ for $\gamma$ defined in (i).
\end{enumerate}
\end{prop}

In our proof of Proposition \ref{simplecurve1}, we will need the following two lemmas.

\begin{lemma}\label{lowerlemma}
Let $c \in \mathbb{C}$ with Arg$(c) \in (0, \frac{\pi}{2}]$, 
and let $\lambda(t) = c\sqrt{1-t}$.
If there exists some time $\tau \in [0, T_z \wedge 1)$ with 
$$\text{Im }g_\tau(z) \geq \text{Im } \l(\tau),  $$
then $\text{Im }g_t(z) \geq \text{Im }\l(t)  $ for all $t \in [\tau,T_z \wedge 1). $
\end{lemma}

\begin{proof}
Let $\phi(t) = \text{Im} \,g_t(z) - \text{Im}\,\l(t)$ for $t \in [\tau,T_z \wedge 1). $
Then
\begin{align*}
\phi'(t) 
  &= \text{Im}\left( \frac{2}{g_t(z) - \l(t)} \right) +  \frac{\text{Im}(c)}{2\sqrt{1-t}} \\
  & = -2\frac{\text{Im }g_t(z) - \text{Im }\l(t)}{|g_t(z)-\l(t)|^2} +  \frac{\text{Im}(c)}{2\sqrt{1-t}} \\ 
  & = -2\frac{\phi(t)}{|g_t(z)-\l(t)|^2} +  \frac{\text{Im}(c)}{2\sqrt{1-t}}.
\end{align*}
Recall that $\phi(\tau) \geq 0$.  If there is some time $t_0 \in [\tau,T_z \wedge 1)$ with $\phi(t_0) = 0$, then $\phi'(t_0) > 0$, and $\phi$ is increasing at $t_0$.  
This implies that $\phi(t) \geq 0$ for all $t \in [\tau,T_z \wedge 1)$.

\end{proof}

\begin{lemma} \label{HullLocation}
Let $\l(t)$ be a driving function on the closed interval $[0,T]$.
Define the vertical strip $V_\l$ and horizontal strip $H_\l$ by
\begin{align*}
 &V_{\l} = \{z\in\mathbb{C}: \min_{t\in[0,T]}\text{Re}\,\l(t) \leq \text{Re}(z) \leq \max_{t\in[0,T]}\text{Re}\,\l(t)\} \;\; \text{ and }\\
  &H_{\l} = \{z\in\mathbb{C}: \min_{t\in[0,T]}\text{Im}\,\l(t) \leq \text{Im}(z) \leq \max_{t\in[0,T]}\text{Im}\,\l(t)\}.
\end{align*}
Then
$$L_{T,\l} \subset V_{\l} \; \; \text{ and } \;\;  R_{T,\l} \subset H_{\l}. $$
\end{lemma}

\begin{proof}
If $w\in\mathbb{C}\setminus V_{\l}$, then either Re$(w) > \max_{t\in[0,T]}\text{Re}\,\l(t)$ or Re$(w) < \min_{t\in[0,T]}\text{Re}\,\l(t)$. 
We will consider the first case; the second case is similar.
Note that
$$\partial_t \text{Re}\,g_t(z) 
= 2\frac{\text{Re}\,g_t(z) - \text{Re}\,\l(t)}{|g_t(z)-\l(t)|^2}. $$
Therefore,
$\partial_t \text{Re} \, g_t(w) > 0$ as long as $g_t(w)$ remains to the right of $V_\l$.
This means it is impossible for $g_t(w)$ to move left in order to enter $V_\l$.
Since $\l(t) \in V_\l$, we must have that  $g_t(w) \neq \l(t)$ for all $t \in[0,T]$
which implies that $w \notin L_{T,\l}$.  Hence $L_{T,\l} \subset V_{\l}$. 

Now, let $\Tilde{\l}(s) = -i\l(T-s)$. 
Applying the first result to $\Tilde{\l}$ gives that $L_{T,\Tilde{\l}} \subset V_{-i\l}$. By the duality property, 
$$R_{T,\l} = iL_{T,\Tilde{\l}} \subset iV_{-i\l} = H_{\l}. $$
\end{proof}

\begin{proof}[Proof of Proposition \ref{simplecurve1}]

By the reflection property, we may assume that  Arg$(c) \in [0, \frac{\pi}{2}]$.  
Further, since  the behavior when $c \in \mathbb{R}$ is well understood, we may assume Arg$(c) \in (0, \frac{\pi}{2}]$.

In this proof, we will be considering the hulls generated by $\lambda |_{[a,b]}$ for different time intervals $[a,b]$, and so we introduce a more convenient notation for this.
Let $L_{[a,b]}$ be the left hull $L_{b-a, \l(a+ \cdot)}$ which is generated by $\l(a+ \, \cdot)$ at time $b-a$,
and let $R_{[a,b]}$ be the corresponding right hull. 
Let $g_{[a,b]} : \mathbb{C} \setminus L_{[a,b]} \to \mathbb{C} \setminus R_{[a,b]}$ be the associated conformal map generated by \eqref{cle}.

Since $\l $ is differentiable on $[0,s]$, there exists $\delta>0$ so that $\l$ is Lip$(1/2)$ with norm at most $1/3$ on any subinterval of $[0,s]$ with length $\delta$.  
Subdivide $[0,s]$ into intervals  $[t_{k-1}, t_{k}]$ with $0=t_0 < t_1 < \cdots < t_n = s$ and $ t_{k}-t_{k-1} \leq \delta$.
By Theorem \ref{HuyTheorem}, the  hulls $L_{[t_{k-1}, t_{k}]} $ and $R_{[t_{k-1}, t_{k}]}$ are simple curves.
We wish to follow a standard argument in the real-valued case which uses the concatenation property to build $L_s = L_{[0,s]}$ out of the conformal images of the curves $L_{[t_{k-1}, t_k]}$.  However, we will need to take a little more care with this argument in the complex-valued case.

We will use induction to prove that $L_{[t_k,s]}$ is a simple curve, starting with $k=n-1$ and decreasing to $k=0$.
As mentioned above, Theorem \ref{HuyTheorem} gives  the base case that $L_{[t_{n-1},s]}$ is a simple curve.
For our inductive step, we assume that  $L_{[t_k,s]}$ is a simple curve, and we must show that $L_{[t_{k-1},s]}$ is a simple curve.
Assume for the moment that $L_{[t_{k},s]} \setminus \{\l(t_{k})\}$ does not intersect $R_{[t_{k-1}, t_k]}$. 
Then the concatenation property will imply that 
$$ L_{[t_{k-1}, s]} = L_{[t_{k-1}, t_k]} \cup g_{[t_{k-1}, t_k]}^{-1} \left( L_{[t_k,s]} \setminus \{ \l(t_{k}) \} \right). $$
Since $L_{[t_k,s]} \setminus \{ \l(t_{k}) \}$ is a union of two simple curves in $\mathbb{C} \setminus R_{[t_{k-1}, t_k]}$ that approach $\l(t_k)$, 
the conformal image $g_{[t_{k-1}, t_k]}^{-1} \left( L_{[t_k,s]} \setminus \{ \l(t_{k}) \} \right)$ will be two simple curves in $\mathbb{C} \setminus L_{[t_{k-1}, t_k]}$ that approach the two tips of $L_{[t_{k-1}, t_k]} $.  
Gluing these with $L_{[t_{k-1}, t_k]} $ gives that $L_{[t_{k-1},s]}$ is a simple curve, as illustrated in Figure \ref{InductionStep}.

%% FIGURE %%
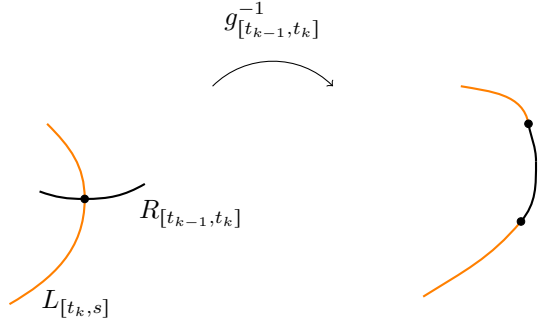
\begin{figure}
\centering
\begin{tikzpicture}

\draw[thick,orange] (0,0) to [out=90,in=-45] (-0.5,1);
\draw[thick] (0,0) to [out=180,in=-20] (-0.6, 0.1);
\draw[thick,orange] (0,0) to [out=-90,in=30] (-1,-1.4);
\draw[thick] (0,0) to [out=0,in=-150] (0.8, 0.2);
\node[right] at (0.6, -0.2) {$R_{[t_{k-1}, t_k]}$};

\draw[fill] (0,0) circle [radius=0.05];
\node[right] at (-0.7, -1.4) {$L_{[t_k,s]} $};

\draw[->] (1.7,1.5) to [out=45,in=135] (3.3,1.5);
\node[above] at (2.5,2) {$g_{[t_{k-1}, t_k]}^{-1}$};

\draw[thick,black] (6,0.5) to [out=90, in=-80] (5.9,1);
\draw[thick,black] (6,0.5) to [out=-90, in=50] (5.8,-0.3);

\draw[thick, orange] (5.9,1) to [out=100, in=-10] (5,1.5) ;
\draw[thick,orange] (5.8,-0.3) to [out=230, in=32] (4.5,-1.3) ;

\draw[fill] (5.9,1) circle [radius=0.05];
\draw[fill] (5.8,-0.3) circle [radius=0.05];

\end{tikzpicture}
\caption{
An illustration of the concatenation property used in the proof of Proposition \ref{simplecurve1}:  
the curve $L_{[t_{k-1}, s]}$ on the right is a union of $L_{[t_{k-1}, t_k]}$ (shown in black) and $g_{[t_{k-1}, t_k]}^{-1} \left( L_{[t_k,s]} \setminus \{\l(t_k) \} \right)$ (shown in orange).
}\label{InductionStep}
\end{figure}

To establish the first statement of (i),
it remains to show that $L_{[t_{k},s]} \setminus \{\l(t_{k})\}$ does not intersect $R_{[t_{k-1}, t_k]}$ when Re$(\alpha) \neq 0$.
By way of contradiction, assume that Re$(\alpha) \neq 0$
and $L_{[t_{k},s]} \setminus \{\l(t_{k})\}$ does intersect $R_{[t_{k-1}, t_k]}$.
Let $\sigma > t_k$ be the first time that $L_{[t_{k},\sigma]} \setminus \{\l(t_{k})\}$ intersects $R_{[t_{k-1}, t_k]}$, 
and let $\tau< t_k $ be the largest time so that $L_{[t_{k},\sigma]} \setminus \{\l(t_{k})\}$ intersects $R_{[\tau, t_k]}$.
Then $L_{[t_{k},\sigma]}$ contains a curve that starts at $\l(t_k)$ and ends at one of the tips of $R_{[\tau, t_k]}$.  (To visualize this, think of adapting the left picture in Figure \ref{InductionStep} so that one tip of the orange curve intersects one tip of the black curve.)
The conformal image of this curve under $g_{[\tau, t_k]}$ is a curve that starts at one tip of $L_{[\tau, t_k]}$ and ends at $\l(\tau)$.
Since the concatenation property implies that
$$ L_{[\tau, \sigma]} = L_{[\tau, t_k]} \cup g_{[\tau, t_k]}^{-1} \left( L_{[t_k, \sigma]} \setminus \{ \l(t_{k}) \} \right), $$ 
we find that $L_{[\tau, \sigma]}$ contains a loop that starts and ends at $\l(\tau)$.  (To visualize, imagine adapting the right picture in Figure \ref{InductionStep} so that one tip of the orange curve intersects the middle of the black curve.)

Now recall that $L_{[\tau, \sigma]}$ is the left hull generated by $\l_{[\tau, \sigma]}(t) = c\sqrt{1-\tau - t}$ at time $t=\sigma-\tau$.
Notice that when we apply Brownian scaling to $ \l_{[\tau, \sigma]}$, we obtain our original driving function $\l$, i.e.
\begin{align*}
\hat \lambda(t) &:= \frac{1}{\sqrt{1-\tau}} \l_{[\tau, \sigma]}\left( (1-\tau)t \right) \\
  &=  \frac{1}{\sqrt{1-\tau}} c\sqrt{ 1-\tau - (1-\tau) t } \\ 
  &= c\sqrt{1-t}
\end{align*}
Thus the scaling property implies that $L_{[0, r]}$
for $r= (\sigma-\tau)/(1-\tau)$ is a scaled version of $L_{[\tau, \sigma]}$.
In particular, $L_{[0, r]}$ contains a simple closed curve through $c = \l(0)$.
We can parametrize this curve by $z_t$ satisfying \eqref{TipEquation}
so that as $t \nearrow r$, we have that $z_t \to c$.
Note that this curve can encircle both of the points $A, B$, one of these points, or neither.  Thus the winding number of the curve around $A$ (resp. $B$) can be 1, 0, or $-1$.  If the winding number around both $A$ and $B$ is non-zero, then it must be same for both points.

We fix a branch of the logarithm for the moment and plug $t=r \neq 0$ and $z_t = c$ into \eqref{TipEquation} to obtain
\begin{align*}
 \alpha\log B + \beta\log A + \log\sqrt{1-r} 
 &= \alpha\left[\log(c-A) +i 2\pi p  \right]+ \beta\left[\log(c-B) + i 2\pi q \right] \\ \nonumber
 &= \alpha\left[ \log B +i 2\pi p \right] + \beta\left[\log A + i 2\pi q \right],
 \end{align*}
 where $p, q \in \{-1, 0, 1\}$ and if both are nonzero, then $p=q$.
 Simplifying the above equation yields
 \begin{equation}\label{cc}
    \log \sqrt{1-r} = i 2\pi p \alpha + i 2 \pi q \beta.
 \end{equation}
 The imaginary part of \eqref{cc} gives 
 $$ 0= 2 \pi p \text{Re}(\alpha) + 2 \pi q \text{Re}(\beta). $$
 Note that Re$(\alpha) \neq 0$ by assumption and Re$(\beta) \neq 0$ by Lemma \ref{parameterlemma}.  
 If $p, q$ are both nonzero, then $p=q$ which implies that Re$(\alpha) + \text{Re}(\beta) = 0$.
 However, this contradicts the fact that $\alpha + \beta = 1$.
 Thus one of $p,q$ must be zero, and this immediately implies both must be zero.
Then equation \eqref{cc} becomes
 $$\log\sqrt{1-r} = 0,$$
which yields a contradiction since $r \neq 0$.
This completes the contradiction proof showing that
that $L_{[t_{k},s]} \setminus \{\l(t_{k})\}$ does not intersect $R_{[t_{k-1}, t_k]}$ when Re$(\alpha) \neq 0$
and subsequently establishes the first statement of (i).
The last statement of (i) follows from Tran's Theorem \ref{HuyTheorem} and the concantenation property.  

Further, Tran's Theorem \ref{HuyTheorem} and the concatenation property show  that 
there is $\epsilon >0$ so that for every $z$ with $T_z \leq s$ 
 when $t \in [T_z-\epsilon, T_z)$ then $g_{t}(z)$ is in the simple curve $L_{[t, t+\epsilon]} =: \gamma_t[-\epsilon,\epsilon]$.
Thus we can define the upper left hull $L_s^+$ as all the points $z \in L_s$ with  
$g_t(z)$ in the upper curve $\gamma_t[0,\epsilon]$ for $t$ close enough to $T_z$.
Similarly  the lower left hull $L_s^-$ is all  $z \in L_s$ 
with $g_t(z) \in \gamma_t[-\epsilon, 0]$ for $t$ close enough to $T_z$.
Statement (ii) will follow once we show that 
$L^-_{[\tau,s]} \setminus \{\l(\tau)\}$ does not intersect $R_{[0, \tau]}$ for any $\tau \in (0,s)$.
From Tran's work in \cite{T}, the lower curve $\gamma_t[-\epsilon, 0)$ will be in the open halfplane below the horizontal line $y = \text{Im }\l(t)$.
Therefore for $z \in L_{[\tau,s]}^-$ and $t$ close enough to $T_z$, we must have that $\text{Im } g_t(z) < \text{Im } \l(t)$.
Lemma \ref{lowerlemma} then implies that $\text{Im } z <\text{Im }\l(\tau)$.
Hence $L^-_{[\tau,s]}\setminus \{\l(\tau)\}$ is in the open halfplane below the horizontal line $y = \text{Im } \l(\tau)$. 
From Lemma \ref{HullLocation}, 
$R_{[0, \tau]}$ is in the closed halfplane above this line.  
Therefore $L^-_{[\tau,s]} \setminus \{\l(\tau)\}$ does not intersect $R_{[0, \tau]}$ for any $\tau \in (0,s)$.

\end{proof}

As a consequence to this proposition, when Re$(\alpha) \neq 0$ and $t <1$, 
we know that there are exactly two points in $L_t \setminus \bigcup_{s \in [0,t)} L_s$ and these are $\gamma(t)$ and $\gamma(-t)$.

\subsection{Behavior as $t$ approaches 1 }

From Proposition \ref{simplecurve1} we know that for $t <1$, the left hull is a simple curve  when Re$(\alpha) \neq 0$.
We now address the behavior of this curve as $t$ approaches 1.

\begin{prop}\label{AB}
Let $c \in \mathbb{C}$ with Arg$(c) \in [0,\pi/2]$, let $\l(t) = c \sqrt{1-t}$,
and let $\gamma $  be defined as in Proposition \ref{simplecurve1}. 
\begin{enumerate}
\item[(i)] When Re$(\alpha) > 0$, the curve  $\gamma(t)$ approaches $A$ as $t \to 1$,
and when Re$(\alpha) < 0$, the curve  $\gamma(t)$ approaches $B$ as $t \to 1$.
\item[(ii)] As $t \to 1$, the curve  $\gamma(-t)$ approaches $B$. 
\end{enumerate}
Moreover, when Arg$(c) \in(0,\pi/2)$, then $\gamma(t),\gamma(-t)$ approach their limit points through an infinite spiral.
\end{prop}

First we show that the limits of $\gamma(t), \gamma(-t)$ are well-defined as $t$ approaches 1.

\begin{lemma}\label{CurveAtTime1}
Let $c \in \mathbb{C}$, let $\l(t) = c \sqrt{1-t}$,
and let $\gamma = \gamma_c$  be defined as in Proposition \ref{simplecurve1}. 
\begin{enumerate}
\item[(i)] When Re$(\alpha) \neq 0$, then $\displaystyle \lim_{t \nearrow 1} \gamma_c(t)$ exists and is continuous in $c$.
\item[(ii)] If Re$(\alpha) \neq 0$ or if Arg$(c) \in (0, \pi/2)$, then $\displaystyle \lim_{t \nearrow 1} \gamma_c(-t)$ exists and is continuous in $c$.
\end{enumerate}
\end{lemma}

We will use the notion of holomorphic motion to prove this result.
A holomorphic motion of $E \subset \mathbb{C}$ is a map $h: \mathbb{D} \times E \to \mathbb{C}$ with the following three properties:
\begin{enumerate}
\item For any fixed $w \in E$, the map $\zeta \mapsto h(\zeta, w)$ is holomorphic in $\mathbb{D}$.
\item For any fixed $\zeta \in \mathbb{D}$, the map $w \mapsto h(\zeta, w)$ is an injection.
\item The mapping  $w \mapsto h(0, w)$ is the identity on $E$.
\end{enumerate}
The following important theorem about holomorphic motions is due to \cite{MSS} and \cite{S}.
\begin{thm}\label{holomotion}
If $h: \mathbb{D} \times E$ is a holomorphic motion, then $h$ has an extension to $H: \mathbb{D} \times \mathbb{C} \to \mathbb{C}$ so that
$H$ is a holomorphic motion of $\mathbb{C}$, each map $w \mapsto  H(\zeta, w)$ is quasisymetric, and $H$ is jointly continuous in $(\zeta, w)$.
\end{thm}

\begin{proof}[Proof of Lemma \ref{CurveAtTime1}]

Let $\Omega = \{c \in \mathbb{C} \setminus \{\pm 4\} \, : \, \text{Re}(\alpha) \neq 0\}$, which is pictured in Figure \ref{ReAlphaIs0}, and let $D$ be a disc of radius $r$ centered at $c_0$ with $D \subset \Omega$.
For $c \in \Omega$,
the top curve of the left hull
$\gamma_c[0,1)$  is a simple curve by Proposition \ref{simplecurve1}.
Let  $s \in (0,1)$ and let $E =  \gamma_{c_0}[s,1)$.
For $w \in E$, let $\sigma(w) = \gamma_{c_0}^{-1}(w)$, i.e. $\sigma(w)$ is the unique time $t \in [s, 1)$ so that $\gamma_{c_0}(t) = w$, which is well-defined since $\gamma_{c_0}$ is a simple curve.
Define $h: \mathbb{D} \times E \to \mathbb{C}$ by
$$h(\zeta, w) = \gamma_{c_0+r\zeta}\left( \sigma(w) \right).$$
We will show that $h$ is a holomorphic motion of $E$.
Note that for fixed $\zeta \in \mathbb{D}$, the map $w \mapsto h(\zeta, w)$ is injective since $\gamma_c$ is a simple curve for each $c \in \Omega$, verifying property 2.
Property 3 follows from the definition of $\sigma$:
$$h(0,w) = \gamma_{c_0} ( \sigma(w))   = w.  $$

To show that $h$ is a holomorphic motion of $E$, we must show that  $ h(\zeta, w)$ is holomorphic in $\zeta$ for fixed $w \in E$.
This will follow from showing that $\gamma_c(t)$ is holomorphic in $c$ for fixed $t \in [s, 1)$.  
Recall from \eqref{TipEquation} that $ \gamma_c(t)$ satisfies
\begin{equation*}
    \alpha\log B + \beta\log A + \log\sqrt{1-t} = \alpha\log(\gamma_c(t)-A) + \beta\log(\gamma_c(t)-B),
\end{equation*}
Note that $A, B, \alpha, $ and $\beta$ are holomorphic in $c$, for $c$ in any disk that avoids $\pm4$. (In fact, $A' = \beta$ and $B'= \alpha$.)  This implies that the derivative of the lefthand side of the above equation ($\partial_c LHS$) is well-defined.  A short computation yields
\begin{align*}
\partial_c \gamma_c(t) = \frac{(\gamma_c(t) - A)( \gamma_c(t) - B)}{\gamma_c(t) - c} &\left[ \partial_c LHS 
          - \alpha' \log(\gamma_c(t) -A) - \beta'\log(\gamma_c(t) -B) \right] \\
          &+  \alpha\beta \frac{2 \gamma_c(t) -c}{\gamma_c(t) -c}. 
\end{align*}          
Since $t \in [s, 1)$ and $\gamma_c(\cdot)$ is only equal to $c$ at time $0$, we have that $\partial_c \gamma_c(t)$ is well defined for $c \in \Omega$.
Thus $h$ is a holomorphic motion of $E$.

Theorem \ref{holomotion} implies that $h$ has an extension to $H: \mathbb{D} \times \mathbb{C} \to \mathbb{C}$ so that $H$ is jointly continuous in $(w, \zeta)$.  
This implies that $\gamma_c(1) := \displaystyle \lim_{t \nearrow 1} \gamma_c(t)$ exists and is continuous in $c$ for $c \in D$.  
Since this is true for all discs $D \subset \Omega$, we have that $\gamma_c(1)$ exists and is continuous in $c$ for all $c \in \Omega$.

We can apply the same proof to the lower curve to show that 
$\gamma_c(-1)$ exists and is continuous in $c$ for all $c \in \Omega$.
Further, we can extend this result to $c$ with Re$(\alpha)=0$ and Arg$(c) \in (0,\pi/2)$ using
Proposition \ref{simplecurve1}(ii).

\end{proof}

We are now ready to prove Proposition \ref{AB}.

\begin{proof}[Proof of Proposition \ref{AB}.]
We begin by identifying $A$ and $B$ as the possible limits for $\gamma(t)$ and $\gamma(-t)$ as $t \to 1$.
Let $z_t = \gamma(t)$, and assume that Re$(\alpha) \neq 0$.
By Lemma \ref{CurveAtTime1}, we know that the limit of $z_t$ as $t \to 1$ exists.
Taking the real part of both sides of equation \eqref{TipEquation} yields
\begin{align*}
\text{Re}&(\alpha)\log(|B|) - \text{Im}(\alpha)\text{Arg}(B)+ \text{Re}(\beta)\log(|A|)-
 \text{Im}(\beta)\text{Arg}(A) + \log\sqrt{1-t} \\
 &=\text{Re}(\alpha)\log|z_t-A|-\text{Im}(\alpha)\text{Arg}(z_t-A)+\text{Re}(\beta)\log|z_t-B|-\text{Im}(\beta)\text{Arg}(z_t-B) \text{.} 
 \end{align*}
Since the left-hand side of the above equation diverges to $-\infty$ as $t \to 1$, at least one of the four terms on the right side must also diverge to $-\infty$. 
Note that if either of the first two terms approaches $-\infty$ as $t \to 1$, then $z_t \to A$.
If the either of the last two terms approaches $-\infty$, then $z_t \to B$.
The same proof applies when $z_t = \gamma(-t)$.

Next we determine when $\gamma(1)$ equals $A$ and when  equals $B$. 
By Lemma \ref{CurveAtTime1}, we know that $\gamma(1)$ is continuous in $c$ 
on the set $\{ c \in \mathbb{C} \setminus \{\pm 4\} \, : \, \text{ Arg}(c) \in [0,\pi/2] \text{ and } \text{Re}(\alpha) \neq 0 \}$.  
This set has two connected components (see Figure \ref{ReAlphaIs0}); the component corresponding to Re$(\alpha)>0$ contains the real interval (0,4), and the component corresponding to Re$(\alpha)<0$ contains $(4, \infty)$.
From the real-valued case in Proposition \ref{KNKprop}, we know that 
$\gamma_c(1)=A$ for $c \in (0,4)$
and $\gamma_c(1)=B$ for $c \in (4, \infty)$.
By Lemma \ref{parameterlemma}, $A$ and $B$ are in separate regions (which only intersect for $c=4$.)
Therefore, by the continuity in $c$, we have that 
$\gamma_c(1) = A$ when Re$(\alpha) >0$ 
and $\gamma_c(1) = B$ when Re$(\alpha) <0$. 
The same proof applies to show that $\gamma_c(-1) = B$ 
for all $c$ with Arg$(c) \in [0, \pi/2]$.

Our last step is to identify the spiraling behavior of $z_t = \gamma(t)$ or $\gamma(-t)$, when Arg$(c) \in (0,\pi/2)$.
We first consider the case when $z_t \to A$ (i.e. 
$z_t = \gamma(t)$ and Re$(\alpha)>0$).
Taking the imaginary parts of both sides of equation \eqref{TipEquation} gives 
\begin{align*}
\text{Re}&(\alpha)\text{Arg}(B)+\text{Im}(\alpha)\log|B| + \text{Re}(\beta)\text{Arg}(A)+\text{Im}(\beta)\log|A| \\
    &=\text{Re}(\alpha)\text{Arg}(z_t-A) +\text{Im}(\alpha)\log|z_t-A|+\text{Re}(\beta)\text{Arg}(z_t-B)+\text{Im}(\beta)\log|z_t-B|\text{.}
\end{align*}
The left-hand side is independent of $t$, and so the right-hand side must remain bounded as $t \to 1$.
However, Im$(\alpha)\log|z_t-A| \to - \infty$, since Im$(\alpha) >0$ by Lemma \ref{parameterlemma}.
Thus there must be at least one term in the right-hand side that approaches $+ \infty$ as $t\to 1$. 
The terms Re$(\beta)\text{Arg}(z_t-B)$ and Im$(\beta)\log|z_t-B|$ must remain bounded since  $z_t \to A \neq B$. 
Therefore $\text{Arg}(z_t-A) \to + \infty$ as $t \to 1$, which implies that $\gamma(t)$ spirals counterclockwise around $A$ as $t \to 1$.
For the case when $z_t \to B$, we will use that 
$\text{Im}(\beta)<0$ and $\text{Re}(\beta)>0$ by Lemma \ref{parameterlemma}.
Then arguing as above, $z_t \to B$ implies that $\text{Im}(\beta)\log|z_t-B| \to +\infty$.
This further implies that $\text{Arg}(z_t-B) \to -\infty$, which means that $z_t$ spirals clockwise around $B$ as $t \to 1$.

\end{proof}

\subsection{Time-1 hulls}
To finish our analysis when Re$(\alpha) \neq 0$, it remains to determine the time-1 left hulls $L_1$.
We begin by showing that $A,B \in L_1$ and, when Re$(\alpha) <0$, the interior of $\gamma$ is also in $L_1$.

\begin{lemma}\label{TargetPointsInHull}
Let $c \in \mathbb{C}$ and let
 $\l(t) = c\sqrt{1-t}$.
 Then $A,B \in L_{1}$. 
\end{lemma}

The proof will further show that $T_A= T_B = 1$, i.e. $A,B$ are added to the hull at time 1.

\begin{proof}
Define  $w_A(t) = A\sqrt{1-t}$ and $w_B(t) = B\sqrt{1-t}$.
We will show that these
are solutions to the Loewner equation with driving function $\l$.
Since $AB=4$ by \eqref{AB4},
$$\partial_t w_A(t)= -\frac{A}{2\sqrt{1-t}}
= -\frac{2}{B\sqrt{1-t}} = \frac{2}{(A-c)\sqrt{1-t}}=\frac{2}{w_A(t)-\l(t)} .$$
Since \eqref{cle} has a unique solution starting from point $A$, 
we must have that $g_t(A) = w_A(t)$. 
In other words, the Loewner flow of $A$ with driving function $\lambda$ is given by by $w_A(t)$.
A similar computation shows $g_t(B) = w_B(t)$.
We notice that, for $t < 1$, $w_A(t), w_B(t) \neq \lambda(t)$, but at time $t = 1$, 
$$w_A(1) = w_B(1) = \lambda(1) = 0, $$
showing that $A, B \in L_{1}$.
\end{proof}

If Re$(\alpha) < 0$, then $\gamma(1) = \gamma(-1)$ by Proposition \ref{AB} and the reflection property, and consequently $\gamma[-1,1]$ is a simple closed curve.  Let Int$(\gamma)$ be the bounded component of $\mathbb{C}\setminus \gamma[-1,1]$.

\begin{lemma} \label{interior}
Let $c \in \mathbb{C}$ with Re$(\alpha) < 0$, and let $\l(t) = c\sqrt{1-t}.$
Then $\gamma[-1,1] \cup \text{Int}(\gamma)$ are contained in  $L_1$.
\end{lemma}

\begin{proof}
Since we already know that $\gamma[-1,1]$ is contained in $L_1$, we need only prove that Int$(\gamma) \subset L_{1}$. 
Fix $t_0 \in (0,1)$.  Then for $t \in [0, 1-t_0]$,
$$\l(t_0+t) = c\sqrt{(1-t_0) -t} = \sqrt{1-t_0} \cdot c\sqrt{1-t/(1-t_0)} = \sqrt{1-t_0} \cdot \l\left( t/(1-t_0) \right).$$
The scaling property implies that 
$$L_{t, \l(t_0+\cdot)} = \sqrt{1-t_0}\, L_{t/(1-t_0), \l}.$$
Plugging in $t=1-t_0$ yields
$$L_{1-t_0, \l(t_0+\cdot)} = \sqrt{1-t_0}\, L_{1, \l},$$
and so
$$\text{diam}\,L_{1-t_0,\l(t_0+\cdot)} = \sqrt{1-t_0}\text{ diam}\,L_{1,\l}. $$

Now, fix any $w \in \text{Int}(\gamma)$. Then $w \notin L_{t,\l}$ for any $t < 1$, so $g_{t_0}(w)$ is well-defined. 
Under the conformal map $g_{t_0}$, 
the curve $\gamma[-1,-t_0]\cup \gamma[t_0,1]$ is mapped to a loop beginning and ending at $\l(t_0)$, 
and $g_{t_0}(w)$ must lie in the interior of this loop.
However, by the concatenation property this loop will be a subset of
$L_{1-t_0, \l(t_0+\cdot)}.$
Therefore there exists $k>0$ so that 
$$|g_{t_0}(w)| \leq |\l(t_0)| + \text{diam}\, L_{1-t_0, \l(t_0+\cdot)} \leq k \sqrt{1-t_0}.$$ 
Thus
$$\displaystyle \lim_{t_0 \nearrow 1}g_{t_0}(w) = 0 = \l(1), $$
and so $w \in L_{1,\l}$.
\end{proof}

It remains to address the question of whether there are any additional points that are added to hull at time 1.  
Since we do not have a well-developed theory yet in the complex setting to rule this out, we will do this in a more hands-on way with the following lemma, which we will prove in Section 4.
Let $\hat{\mathbb{C}} = \mathbb{C} \cup \{\infty \}$ be the Riemann sphere,
and define a {\it $t$-accessible} point to be a point $z_0$ in 
$\displaystyle L_t \setminus \bigcup_{s \in [0,t)} L_s$ so that there exists a curve $\eta$ in $\mathbb{C} \setminus L_t$ ending at $z_0$.

\begin{lemma}\label{accessiblepoints}
Let $c \in \mathbb{C}$, and let $\l(t) = c\sqrt{1-t}.$
There are at most two 1-accessible points and $\hat{\mathbb{C}} \setminus L_1$ is simply connected.
\end{lemma}

We are now able to prove the first two parts of Theorem \ref{familyTHM1}.

\begin{proof}[Proof of Theorem \ref{familyTHM1} (i)-(ii).]
Assume that Re$(\alpha) \neq 0$.
Then Proposition \ref{simplecurve1} gives that $L_t$ is a simple curve $\gamma[-t,t]$ for $t<1$.
From Proposition \ref{AB}, Lemma \ref{TargetPointsInHull}, and the reflection property, we can extend $\gamma$ continuously to $[-1,1]$, and the endpoints $\gamma(1), \gamma(-1)$ are in $L_1$.  

When Re$(\alpha) >0$, Proposition \ref{AB} shows that the endpoints of $\gamma[-1,1]$ are distinct and gives
the spiraling behavior of this curve for $c \notin i\mathbb{R}$.
Note that if $L_1 = \gamma[-1,1]$, then the distinct endpoints $\gamma(1),\gamma(-1)$ are 1-accessible points.
It is not possible that any additional points could be added to $\gamma[-1,1]$ to form $L_1$ so that there are only two 1-accessible points, as required by Lemma \ref{accessiblepoints}.
Thus $L_1 = \gamma[-1,1]$.

When Re$(\alpha) <0$, Proposition \ref{AB} implies that $\gamma(-1) = \gamma(1)$,  forming a closed loop, and Lemma \ref{interior} shows that the interior of $\gamma$ is also in $L_1$.
If $L_1 = \gamma[-1,1] \cup \text{Int}(\gamma)$, then the endpoint $\gamma(1)=\gamma(-1)$ is a 1-accessible point.
The only way to add additional points to form $L_1$ so that is has at most two 1-accessible points, as required by Lemma \ref{accessiblepoints}, is to add one point in $\mathbb{C} \setminus \left( \gamma[-1,1] \cup \text{Int}(\gamma) \right)$.  However, this would violate the requirement that $\hat{\mathbb{C}} \setminus L_1$ is simply connected,
proving that $L_1 = \gamma[-1,1] \cup \text{Int}(\gamma)$.
\end{proof}

\section{Left hulls driven by $c\sqrt{\tau + t}$} \label{tauplust}

Properties of the hulls driven by functions of the form $c\sqrt{\tau + t}$ are intertwined with properties of the hulls driven by functions of the form $\hat{c}\sqrt{1-t}$. In fact, the duality and scaling properties relate these hulls.
By scaling, we can assume that $\tau <1$, with $s = 1-\tau$.
Then by duality, the right hull $R_s$ driven by $c\sqrt{\tau +t}$ is a rotation of the left hull $L_s$ driven by $-ic\sqrt{1-t}$, and similarly, the right hull $R_s$ driven by $\hat{c}\sqrt{1-t}$ is a rotation of the left hull $L_s$ driven by $ -i\hat{c}\sqrt{\tau+t}$.

\subsection{Implicit solution }

We begin by solving the Loewner equation for an implicit solution, as in the previous case.
Alternately, one could establish this result using Lemma \ref{ImplicitSolutionLemma} along with the duality and scaling properties.

\begin{lemma} \label{ImplicitSolutionLemma2}
For $c \in \mathbb{C}$ with $c \neq \pm 4i$ and $\tau \geq 0$, let  $\l(t) = c\sqrt{\tau+ t}$ be the driving function.  
Set $D = \frac{1}{2}\left[c+\sqrt{c^2+16}\right], \, E = \frac{1}{2}\left[c-\sqrt{c^2+16}\right], \,
\delta = \frac{1}{2}\left[1-\frac{c}{\sqrt{c^2+16}}\right]$,
and $\epsilon = \frac{1}{2}\left[1+\frac{c}{\sqrt{c^2+16}}\right]$.
Then the solution $g_t(z)$ to \eqref{cle} satisfies
\begin{equation} \label{implicitsol2}
\delta \log\left(g_t(z)-D\sqrt{\tau + t}\,\right)+\epsilon \log\left(g_t(z)-E\sqrt{\tau +t}\,\right) 
= \delta\log(z-D\sqrt{\tau}) + \epsilon\log(z-E\sqrt{\tau}).
\end{equation}
Further $\displaystyle z_t \in L_t \setminus \bigcup_{s \in [0,t)} L_s$
satisfies
\begin{equation}\label{tip2}
 \delta \log\left(E\right)+\epsilon \log\left(D\right) +\log\sqrt{\tau+t}
= \delta\log(z_t-D\sqrt{\tau}) + \epsilon\log(z_t-E\sqrt{\tau}).
\end{equation}
\end{lemma}

\begin{proof}

Fix $c \in \mathbb{C}$ with $c \neq \pm 4i$ and $\tau \geq 0$, and set
$$\l(t) = c \sqrt{\tau + t}.$$
Set 
$$W = W(t,z) = \frac{g_t(z)}{\sqrt{\tau + t}}\text{.} $$
Then
\begin{align*}
\partial_t W = \partial_t\left( \frac{g_t(z)}{\sqrt{\tau+t}}\right) 
&= \frac{1}{\sqrt{\tau+t}} \, \frac{2}{g_t(z)-c\sqrt{\tau +t}} -\frac{1}{2} \frac{g_t(z)}{(\tau+t)^{3/2}} \\
&= \frac{1}{\tau+t}\left(\frac{2}{W-c}-\frac{W}{2} \right) \\
&= - \frac{1}{2} \cdot \frac{1}{\tau + t} \left(\frac{W^2 - cW - 4}{W-c}\right), 
\end{align*}
and 
$$\int \frac{W-c}{W^2-cW-4} dW = -\frac{1}{2}\int \frac{1}{\tau + t} dt \text{.}$$
The roots of the denominator $W^2-cW-4$ are 
$$D = \frac{1}{2}\left[c+\sqrt{c^2+16}\right] \;\;\; \text{ and } \;\;\; E = \frac{1}{2}\left[c-\sqrt{c^2+16}\right].$$
The assumption that $c \neq \pm 4i$ guarantees that the roots are distinct.
Then the partial fraction decomposition yields 
$$\int \left( \frac{\delta}{W-D}+\frac{\epsilon}{W-E} \right)dW 
= -\frac{1}{2} \int \frac{dt}{\tau + t} $$
for 
$$
\delta = \frac{1}{2}\left[1-\frac{c}{\sqrt{c^2+16}}\right] \;\;\;
\text{ and } \;\;\; \epsilon = \frac{1}{2}\left[1+\frac{c}{\sqrt{c^2+16}}\right] \text{.}$$
Note that $D+E=c$ and $\delta + \epsilon = 1$.
Therefore
$$\delta\log(W-D) + \epsilon \log(W-E) = -\frac{1}{2}\log(\tau + t) + K\text{,} $$
where $K$ denotes the constant of integration. 
Since $W = g_t(z)/ \sqrt{\tau + t}$ and $\d + \epsilon = 1$, 
we have
\begin{align*}
K &= \d\log\left(\frac{g_t(z)}{\sqrt{\tau + t}}-D\right) 
+ \epsilon \log \left(\frac{g_t(z)}{
\sqrt{\tau + t}}-E\right) + \log\left(\sqrt{\tau + t}\,\right)  \\
&= \delta \log\left(g_t(z)-D\sqrt{\tau + t}\,\right)+\epsilon \log\left(g_t(z)-E\sqrt{\tau +t}\,\right). 
\end{align*}
Plugging in the initial condition $g_0(z) = z$ gives
$$K = \delta\log(z-D\sqrt{\tau}) + \epsilon\log(z-E\sqrt{\tau}), $$
which establishes \eqref{implicitsol2}.
Equation \eqref{tip2} follows from \eqref{implicitsol2}, since $g_t(z_t) = c \sqrt{\tau+t}$ for $\displaystyle z_t \in L_t \setminus \bigcup_{s \in [0,t)} L_s$.

\end{proof}

The parameters $\delta, \epsilon, D$ and $E$ are related to the parameters $\alpha, \beta, A$ and $B$ as follows:
\begin{equation} \label{parameters}
\delta(c) = \alpha(-ic), \;\;\; \epsilon(c) = \beta(-ic), \;\;\;
D(c) = iA(-ic), \;\;\; \text{ and } \;\;\; E(c) = iB(-ic).
\end{equation}
This fits with the relationship from the  duality property that
right hull $R_s$ driven by $ c\sqrt{(1-s)+t}$  is a rotation by $i$ of the left hull $L_s$ driven by $-ic\sqrt{1-t}$. 
Since the behavior of the hulls generated by $\hat{c}\sqrt{1-t}$ 
depends on whether  Re$(\alpha)$ is positive, negative, or zero,
the behavior of hulls generated by $c\sqrt{\tau+t}$ should
depend on whether Re$(\delta)$ is positive, negative, or zero, 
as we will see in the next sections.

%%%
\subsection{Hulls for $\tau >0$}

The following result establishes Theorem \ref{familyTHM2}(iii) about the left hull generated by $c\sqrt{\tau -t}$ with $\tau >0$ and Re$(\delta) \neq 0$. 
We omit the proof, since it follows the same argument as for Proposition \ref{simplecurve1}.

\begin{prop}  \label{simplecurve2}
Let $c \in \mathbb{C}$, let $\tau >0$, let $\delta = \frac{1}{2}\left[ 1 - \frac{c}{\sqrt{c^2+16}}\right]$, 
let $\lambda(t) = c\sqrt{\tau+t}$, and
let $ s > 0$. 
\begin{enumerate}
\item[(i)]When Re$(\delta) \neq 0$, 
the left hull $L_s$ generated by $\l |_{[0,s]}$ is a simple curve.
In particular, $L_s = \gamma[-s,s]$ for a simple curve $\gamma: [-s, s] \to \mathbb{C}$ with
$$\gamma(t) = \lim_{y \to 0^+} g_t^{-1}(\l(t)+iy) \;\; \text{ and } \;\; \gamma(-t) = \lim_{y \to 0^-} g_t^{-1}(\l(t)+iy).$$

\item[(ii)] When Re$(\delta) = 0$ and Arg$(c) \in (0, \frac{\pi}{2})$, 
 the left hull $L_s$ generated by $\l |_{[0,s]}$ can be decomposed into two sets, an upper left hull $L_s^+$ and a lower left hull $L_s^-$
with $L_s^+ \cup L_s^- = L_s$ and $L_s^+ \cap L_s^- = \{c\}$.  
The upper left hull is a simple curve with $L_s^+ = \gamma[0,s]$ for $\gamma$ defined in (i).
\end{enumerate}
\end{prop}

%%%
\subsection{Hulls for $\tau=0$}
In this section we determine the left hull driven by $c\sqrt{t}$, 
establishing Theorem \ref{familyTHM2}(i)-(ii).

When $\tau = 0$, we can simplify equations \eqref{implicitsol2} and \eqref{tip2} in Lemma \ref{ImplicitSolutionLemma2}.
Setting $\tau =0$ in \eqref{implicitsol2} gives
$$\delta \log\left(g_t(z)-D\sqrt{ t}\,\right)+\epsilon \log\left(g_t(z)-E\sqrt{t}\,\right) 
= \log(z),$$
and so
\begin{equation}\label{inverse}
z =  \left( g_t(z)-D\sqrt{t} \right)^{\d}\left(g_t(z)-E\sqrt{t}\right)^{\epsilon},
\end{equation}
yielding an explicit formula for the inverse.
Using this equation with  $g_t(z_t) = c\sqrt{t}$ 
(and few step of algebra) gives 
\begin{equation}\label{sqrtttips}
z_t = E^\delta D^\epsilon \sqrt{t} \;\;\; \text{ for }  z_t \in L_t \setminus \bigcup_{s \in [0,t)} L_s.
\end{equation}
The form of this equation may obscure the fact that there is an implicit branch choice.  As a result, there may be multiple points $z_t$ that satisfy this.
However, we immediately see that the hull $L_t$ is a union of  line segments that connect the origin to a tip point satisfying \eqref{sqrtttips}.

It remains to resolve the question of how many line segments constitute $L_t$.
By the scaling property, we can assume that $t=1$, and by the reflection property, we can assume that Arg$(c) \in [0,\pi/2]$.
Then from \eqref{inverse}, $g_1^{-1} : \mathbb{C} \setminus R_1 \to \mathbb{C} \setminus L_1$ is given by
$$g_1^{-1}(z) = \left( z - D \right)^{\d}\left( z - E \right)^{\epsilon}. $$
Recall that we can extend $g_1$ to the point at infinity so that $g_1^{-1}$ is conformal on $\hat{\mathbb{C}} \setminus R_1$
(where $\hat{\mathbb{C}} = \mathbb{C} \cup \{ \infty \}$.)
However, looking at the formula for $g_1^{-1}$ and using what we know about logarithms, it is possible that $g_1^{-1}$ is defined on an even larger domain.
In particular, there exists a well-defined branch for this function defined on a simply connected domain $\Omega$ containing  $\hat{\mathbb{C}} \setminus R_1$,
and the complement of $\Omega$ (i.e. the connected ``branch-cut" set containing $D$ and $E$) must be a subset of $R_1$.

We wish to determine how many different limits are possible for $g_1^{-1}(z)$ as  $z$ approaches $c$ along a curve in $\Omega$.  
Let $\eta_1$ be a curve which approaches $c$ in $\Omega$ from above.
(We know $\eta_1$ must exist, since the halfplane above the horizontal line $y=\text{Im}(c)$ is contained in $\mathbb{C} \setminus R_1$ by Lemma \ref{HullLocation}.)
Let $\eta_2$ be another curve that approaches $c$ within $\Omega$.
Connect $\eta_1$ and $\eta_2$ to form a simple closed loop in $\Omega \cup \{c\}$ beginning and ending at $c$.  
The interior of this loop may contain neither $D$ and $E$, both $D$ and $E$, or only one of the points. 
In the first case, there will be no change in Arg$(z-D)$ and Arg$(z-E)$ from the beginning to the end of the loop, and hence the limit of $g_1^{-1}(z)$ along $\eta_1$ and $\eta_2$ will be the same.  
The second case is equivalent to the first by homotopy, since $g_1^{-1}$ is conformal at infinity (or one can compute directly that the limit will be the same using the facts that $\delta + \epsilon = 1$ and the change of Arg$(z-D)$ and Arg$(z-E)$ along the loop will be the same value of either $2\pi$ or $-2\pi$ for both.)
We now consider the third case when the loop surrounds only one of $D$ or $E$.  
There is only one direction (clockwise or counterclockwise) that the loop may go around  $D$ while excluding $E$ and then a loop that goes around $E$ while excluding $D$ must go in the opposite direction.
These two cases are homotopy equivalent and we only need consider one of them (or they can be computed directly to show that they give the same limit).
Let's assume that the curve goes around $D$ while excluding $E$.
Then 
\begin{equation}\label{twolimits}
 \lim_{z \to c, z \in \eta_2} g_1^{-1}(z) = e^{\pm 2\pi \delta i} \lim_{z \to c, z \in \eta_1}g_1^{-1}(z),  
 \end{equation}
where there is only one possible choice of plus or minus as determined by the orientation of the loop.
Therefore, there are at most two points satisfying \eqref{sqrtttips} and $L_1$ is either one line segment or a union of two line segments emanating from zero.

In a moment we will determine exactly when we obtain  one-segment and two-segment hulls, but first we pause to prove Lemma \ref{accessiblepoints}.  This will finish off the loose ends from Section 3 and will fully establish Theorem \ref{familyTHM1}(i)-(ii).

\begin{proof}[Proof of Lemma \ref{accessiblepoints}.]
Let $z_0$ be a 1-accessible point of $L_{1,c\sqrt{1-t}}$.  
This means that $z_0$ is added to the left hull at time 1 and there is a curve $\eta \in \mathbb{C} \setminus L_{1,c\sqrt{1-t}}$ ending at $z_0$.  
Thus  $g_1(\eta)$ is a curve in $\mathbb{C} \setminus R_{1,c\sqrt{1-t}}$ ending at $g_1(z_0) = 0$.
This shows that 1-accessible points of $L_{1,c\sqrt{1-t}}$ correspond to prime ends at 0 in $\mathbb{C} \setminus R_{1,c\sqrt{1-t}}$.  By the above discussion and the duality property, $R_{1,c\sqrt{1-t}}$ is either one or two line segments emanating from the origin.  Therefore, there are at most two prime ends at 0 in $\mathbb{C} \setminus R_{1,c\sqrt{1-t}}$ and hence at most two 1-accessible points of $L_{1,c\sqrt{1-t}}$.  Further $\hat{\mathbb{C}} \setminus R_{1,c\sqrt{1-t}}$ is simply connected, proving that $\hat{\mathbb{C}} \setminus L_{1,c\sqrt{1-t}}$ must also be simply connected.
\end{proof}

Now that we have established Theorem \ref{familyTHM1}(i)-(ii), we will use it, along with the duality property, to understand the hulls
 $R_{1,c\sqrt{t}}$ when Re$(\delta) \neq 0$.  This in turn will allow us to determine when we obtain one-segment or two-segment hulls for $L_{1, c\sqrt{t}}$.
Recall that the right hull $R_{1, c\sqrt{t}}$  is a rotation by $i$ of the left hull $L_{1, -ic\sqrt{1-t}}$, with the parameters related by \eqref{parameters}.
Therefore, when Re$ \, \delta(c) >0$, 
the hull $R_{1, c\sqrt{t}}$ is a simple curve with endpoints $D,E$ that contains the point $c$.  In this situation, we can create a  loop in $\left( \mathbb{C} \setminus R_1 \right) \cup \{c\}$ starting and ending at $c$ that contains one of $D,E$ but not the other, and so we obtain two segments for $L_{1, c\sqrt{t}}$. 
When  Re$ \, \delta(c) <0$, the hull $R_{1, c\sqrt{t}}$
is a simple closed curve and its interior.
Since this set contains both points $D$ and $E$,  
 it is not possible to create a loop in $\left( \mathbb{C} \setminus R_1 \right) \cup \{c\}$ starting and ending at $c$ that contains one of $D,E$ but not the other.  
This means that $L_{1, c\sqrt{t}}$ consists of only one line segment.
(Alternately, we could argue using the correspondence described in the proof of Lemma \ref{accessiblepoints}, and the fact that $R_{1, c\sqrt{t}}$ contains two $1$-accessible points when Re$ \, \delta(c) >0$ and one such point when Re$ \, \delta(c) <0$.)

 It remains to look at the case when Re$(\delta)=0$, which we will examine further in the next section.
 When Re$(\delta)=0$, the factor in equation \eqref{twolimits} become $e^{\mp 2\pi \text{Im}(\delta)} $, which is purely real.  In other words, if there are two tips, they must lie along the same ray, and consequently the overall effect is a one-segment hull.

\section{Transitional hulls}

In this section, we discuss the left and right hulls generated by $c\sqrt{1-t}$ for $c$ with Re$(\alpha) =0$, 
but we will keep this discussion brief and omit some details. 
Through the duality property, this also provides an understanding of the left and right hulls generated by $\hat{c} \sqrt{\tau + t}$ when Re$(\delta) = 0$.
Since the reflection property allows us to restrict our attention to  the first quadrant, we define the phase transition set
$$P = \{ c \in \mathbb{C} : \, \text{Re}(\alpha) = 0 \text{ and Arg}(c) \in (0, \pi/2) \}.$$

We begin by considering the question of whether $L_t$ could be a simple curve for all $t<1$ when $c \in P$.  If this were true for all $c$ in an open disc $U$ in the first quadrant that has nonempty intersection with $P$, then the holomorphic motion argument used in Lemma \ref{CurveAtTime1} would apply to show that $\gamma_c(1)$ is well-defined and continuous in $c$.  
However, $\gamma_c(1)$ equals
$A(c)$ when Re$(c) >0$ but equals $B(c)$ when Re$(c) <0$, 
which yields a discontinuity for $\gamma_c(1)$ for $c \in P$.
Therefore $L_t$ is not a simple curve for all $t<1$ for at least a dense set of $c \in P$.  
In fact, one can use a more careful holomorphic motion argument to show that this holds for all $c \in P$.

\begin{figure}
\centering
\includegraphics[scale=0.6]{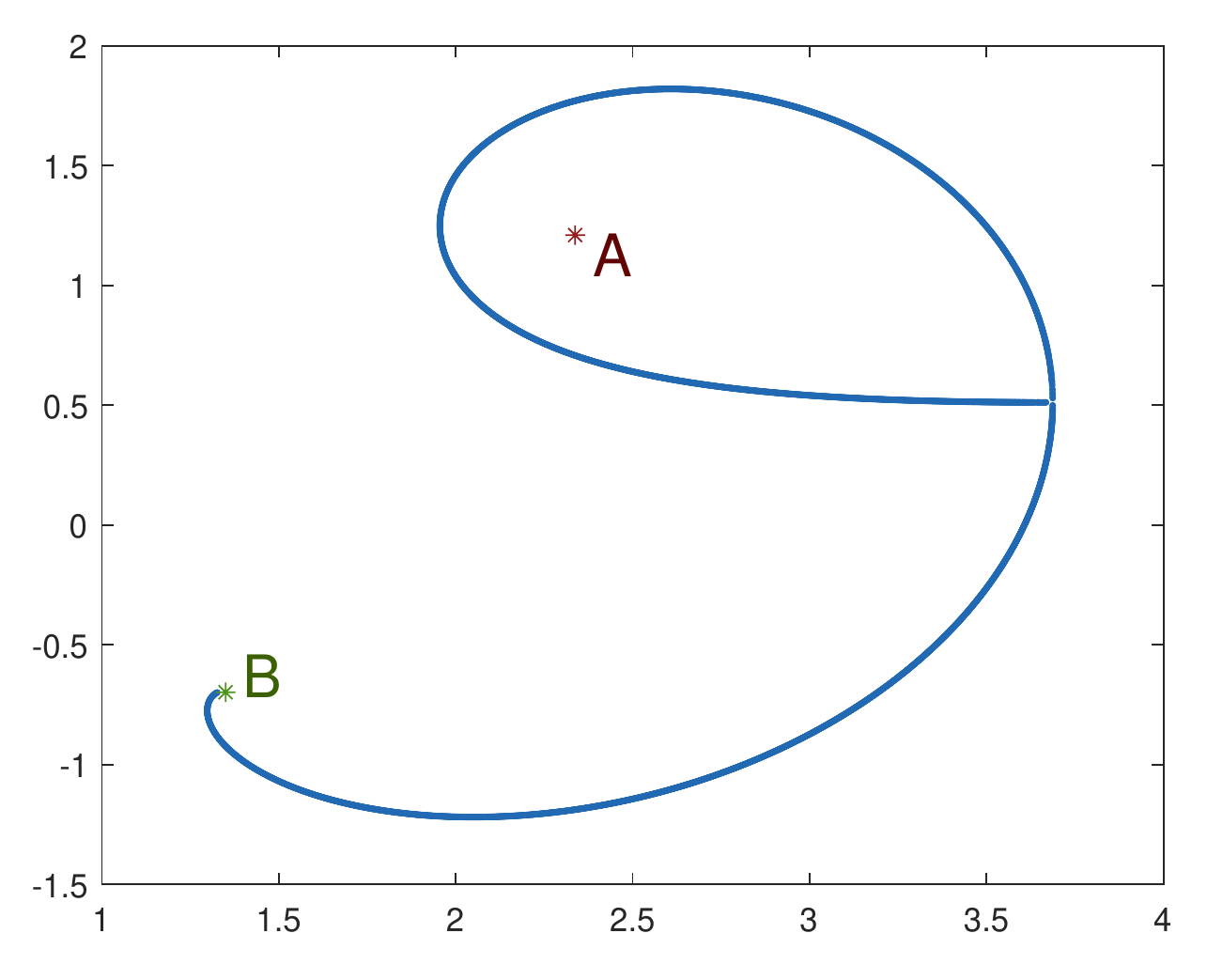}
\caption{The left hull $L_1$ driven by $c\sqrt{1-t}$, where $c \approx  3.687 + 0.511i$ satisfies Re$(\alpha)=0$,
has a lower simple curve and an upper curve that hits back at $c$ prior to time 1.  The points inside the loop are added to the left hull at time 1.
}\label{ReAlphaZeroExample}
\end{figure}

Let $c \in P$ so that $L_t$ is not a simple curve for all $t<1$. 
Then Proposition \ref{simplecurve1} implies that the problem must occur with the upper left hull.
Further from the proof of Proposition \ref{simplecurve1}, we can determine  that there must be some time $t_c< 1$ so that $L_t$ is a simple curve for all $t<t_c$ and $\gamma(t) \to c$ as $t \to t_c$ (and in fact, we can compute $t_c = 1-e^{-4\pi\text{Im}(\alpha)}$).  
In other words, the upper curve of $L_t$ hits back on itself at the point $c$ at time $t_c$.  Thus $L_{t_c}$ is a curve with a loop on top and a simple tail.
Further, the domain for $g_{t_c}$ is no longer connected but has two components.
One can show that the points in the bounded component are captured at time 1,
which implies that
 $\mathbb{C} \setminus L_t$ remains disconnected for $t \in [t_c,1). $
 See Figure \ref{ReAlphaZeroExample} for an image of $L_1$, computed from the implicit solution given in Lemma \ref{ImplicitSolutionLemma}.

Since $\mathbb{C} \setminus L_t$ is disconnected when $t \in [t_c,1)$,
we must also have that its conformal image $\mathbb{C} \setminus R_t$ is also disconnected.
This shows that $R_t$ cannot be a simple curve.
In fact, one can show that for $t \in [t_c, 1)$, the hulls $R_t$  will also consist of a loop and a tail.
As $t \to 1$, the loop of $R_t$ will shrink away, and $R_1$ is a one-segment hull with two tips along the same ray.  The two tips will correspond to the two limits of $g_t(z)$ as $z \to c$ with $z\in \mathbb{C} \setminus L_1$.

\vspace{0.25in}

Department of Mathematics $\cdot$ University of Tennessee $\cdot$ Knoxville, TN 37996


\begin{thebibliography}{abc}

\bibitem[KNK]{KNK}  W.~Kager, B.~Nienhuis, and L.~Kadanoff. Exact solutions for Loewner Evolutions. 
{\it J. Statist. Phys. \bf 115} (2004), 805--822.

\bibitem[L]{L} J.~Lind. 
A sharp condition for the Loewner equation to generate slits. {\it Ann. Acad. Sci. Fenn. Math.  \bf 30} (2005), no 1, 143--158.

\bibitem[MR]{MR} D.E. Marshall, S. Rohde, The Loewner differential equation and slit mappings.
{\it J. Amer. Math. Soc. \bf 18} (2005),  763--778. 

\bibitem[MSS]{MSS} R.~Ma\~n\`e, P.~Sad, and D.~Sullivan.
On the dynamics of rational maps.  {\it Ann. Sci. \`Ecole Norm. Sup. \bf 16} (1983), no 2, 193--217.

\bibitem[Sc] {SLEintro} O.~Schramm, Scaling limits of loop-erased random walks and uniform
              spanning trees.
  {\it Israel J. Math. \bf 118} (2000), 221--288.

\bibitem[Sl]{S} Z.~Slodkowski.
Holomorphic motions and polynomial hulls.  {\it Proc. Amer. Math. Soc.  \bf{111}} (1991), no 2, 347--355.

\bibitem[T]{T} H.~Tran.
Loewner equation driven by complex-valued functions.  arXiv:1707.01023.


\end{thebibliography}
\end{document}